\documentclass{amsart}
\usepackage{amsfonts}
\usepackage{amsmath,amssymb}
\usepackage{amsthm}
\usepackage{amscd}
\usepackage{graphics}
\usepackage{graphicx}
\theoremstyle{remark}{
\newtheorem{Def}{{\rm Definition}}
\newtheorem{Ex}{{\rm Example}}
\newtheorem{Rem}{{\rm Remark}}

}

\newtheorem{Prop}{Proposition}
\newtheorem{Thm}{Theorem}
\newtheorem{Lem}{Lemma}
\begin{document}
\title[Cohomological information of fold maps given by surgerying]{Notes on fundamental fold maps obtained by surgery operations and cohomology rings of their Reeb spaces}
\author{Naoki Kitazawa}
\address{Institute of Mathematics for Industry, Kyushu University, 744 Motooka, Nishi-ku Fukuoka 819-0395, Japan\\
 TEL (Office): +81-92-802-4402 \\
 FAX (Office): +81-92-802-4405 \\
}
\email{n-kitazawa@imi.kyushu-u.ac.jp}
\urladdr{https://naokikitazawa.github.io/NaokiKitazawa.html}
\maketitle
\begin{abstract}
The theory of {\it Morse} functions and their higher dimensional versions or {\it fold} maps on manifolds and its application
 to geometric theory of manifolds is one of important branches of geometry and mathematics. Studies related to this was started in 1950s by differential topologists such as Thom and Whitney and they have been studied actively.

 In this paper, we study fold maps obtained by surgery operations to fundamental
   fold maps, and especially {\it Reeb spaces}, defined as the spaces of all connected components of preimages and in suitable situations inheriting fundamental and important algebraic invariants such as (co)homology groups. Reeb spaces are fundamental and important tools in studying manifolds also in general. The author has already studied about homology groups of the Reeb spaces and obtained several results and in this paper, we study about their cohomology rings for several specific cases, as more precise information. These studies are motivated by a problem that construction of explicit fold maps is important in investigating (the worlds of explicit classes of) manifolds in geometric and constructive ways and difficult. It is not so difficult to construct these maps for simplest manifolds such as standard spheres, products of standard spheres and manifolds represented as their connected sums. We see various types of cohomology rings of Reeb spaces via systematic construction of fold maps.
          

\end{abstract}


\maketitle
\section{Introduction and fundamental notation and terminologies.}
\label{sec:1}

This paper is on studies of differentiable manifolds via differentiable maps into manifolds whose dimensions are smaller such as so-called Morse functions and their higher dimensional versions, mainly, {\it fold} maps. Via explicit fold maps, we understand manifolds in geometric and constructive manners. These understandings are fundamental and important works to do and have been difficult and regarded as new problems. We mainly observe their homology groups and cohomology rings, more precisely, via quotient spaces called {\it Reeb spaces} of the maps. Closely related studies are \cite{kitazawa}, \cite{kitazawa2}, \cite{kitazawa3}, \cite{kitazawa5} and \cite{kitazawa6} and especially, \cite{kitazawa6} are more closely related.

Throughout this paper, maps between manifolds are fundamental objects. Manifolds are smooth (of class $C^{\infty}$) and maps between two manifolds are also smooth (of class $C^{\infty}$) unless otherwise stated. {\it Diffeomorphisms} are always assumed to be smooth.

  {\it Fold} maps are smooth maps regarded as higher dimensional versions of Morse functions and fundamental and important
 tools in studying manifolds by investigating {\it singular points} and {\it singular values} of generic smooth maps : the study is regarded
 as a general version of well-known classical differential topological theory of Morse functions.

A {\it singular} point of a smooth map $c:X \rightarrow Y$ between two smooth manifolds is a point in the source manifold at which the rank of the differential drops. 
The {\it singular set} $S(c)$ of the map $c$ is the set of all singular points of the map.
A {\it singular value} of the map is a point in the image $c(S(c))$ of the singular set: the image is called the {\it singular value set} of the map. The {\it regular value set} of the map is the complement $Y-c(S(c))$ of the singular value set and a {\it regular value} is a point in the regular value set. 

\begin{Def}
\label{def:1}
Let $m$ and $n$ be integers satisfying the relation $m \geq n \geq 1$.
A smooth map from an $m$-dimensional smooth manifold with no boundary into an $n$-dimensional smooth manifold with no boundary is said to be a {\it fold map} if at each singular point $p$, it is represented as
$$(x_1, \cdots, x_m) \mapsto (x_1,\cdots,x_{n-1},\sum_{k=n}^{m-i}{x_k}^2-\sum_{k=m-i+1}^{m}{x_k}^2)$$
 for suitable coordinates and an integer $0 \leq i(p) \leq \frac{m-n+1}{2}$
\end{Def}

Note that throughout this paper, the condition that a smooth map is (locally or globally) represented as another smooth map for suitable coordinates is defined as a condition equivalent to the notion that a smooth map is {\it $C^{\infty}$ equivalent} to another smooth map. However, we do not abuse the latter terminology. For this see also \cite{golubitskyguillemin}.

\begin{Prop}
\label{prop:1}
For a fold map in Definition \ref{def:1}, the following properties hold.
\begin{enumerate}
\item For any singular point $p$, the $i(p)$ as in Definition \ref{def:1} is unique: $i(p)$ is the {\rm index} of $p$. 
\item The set consisting of all singular points of a fixed index of the map is a smooth closed submanifold of dimension $n-1$ with no boundary of the source manifold. 
\item The restriction map to the singular set is a smooth immersion.
\end{enumerate}
\end{Prop}

For fundamental theory of fold maps and more general good maps, see \cite{golubitskyguillemin} for example.

In studies of fold maps and application to algebraic topology and differential topology of manifolds, constructing explicit fold maps is fundamental, important and difficult, where there are several fundamental examples as presented in Example \ref{ex:1}. 

\begin{Ex}
\label{ex:1}
\begin{enumerate}
\item
\label{ex:1.1}
FIGURE \ref{fig:1} represents a Morse function with exactly two singular points, characterizing a homotopy
 sphere whose dimension is not $4$ topologically and the standard sphere of dimension $4$ as the Reeb's theorem shows, and the canonical projection of a unit sphere (of dimension $m$ into ${\mathbb{R}}^n$ with $m \geq n>1$).
\begin{figure}
\includegraphics[width=60mm]{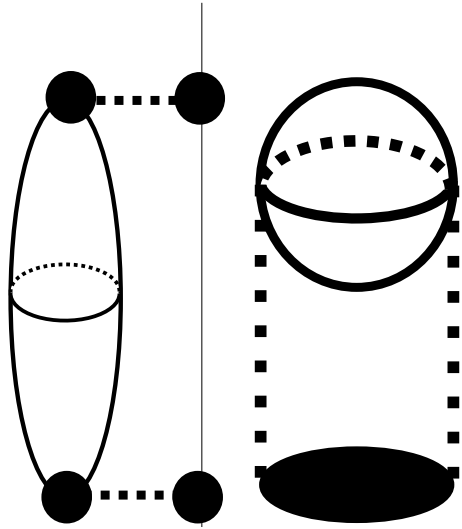}
\caption{A Morse function with exactly two singular points and a canonical projection of an unit sphere: the dots in the left figure are for singular points and singular values and in the right figure, the equator is for the singular set, the black disc is the image, and the boundary of the black disc is the singular value set.}
\label{fig:1}
\end{figure}
\item
\label{ex:1.2}
(Discussed in \cite{saeki} and so on.)
FIGURE \ref{fig:2} represents images of fold maps into the plane such that the restrictions to the singular sets are embedding and that the images of the restriction maps are the boundaries of the images.
\begin{figure}
\includegraphics[width=80mm]{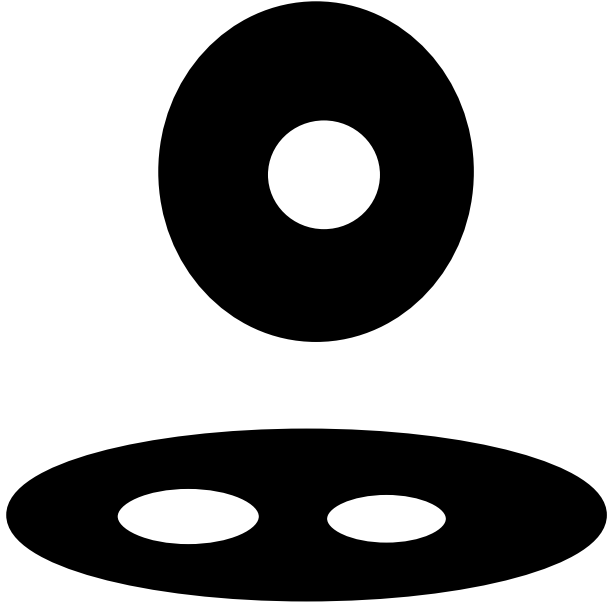}
\caption{Images of fold maps on $S^1 \times S^{m-1}$ ($m \geq 2$) and a manifold represented as a connected sum of two copies of the manifold into the plane such that the restriction maps to the singular sets are embedding and that the images of the restriction maps are the boundaries of the images.}
\label{fig:2}
\end{figure}
\item 
\label{ex:1.3}
(Discussed in \cite{saekisuzuoka} and later in \cite{kitazawa}, \cite{kitazawa2}, \cite{kitazawa4} and so on.)
FIGURE \ref{fig:3} represents a fold map into the plane or ${\mathbb{R}}^n$ ($n \geq 3$) such that the restriction to the singular set is embedding, that the singular set is a disjoint union of two spheres and that the preimage of each regular value is $S^{m-n}$ or a disjoint union of two copies of an ($m-n$)-dimensional homotopy
 sphere $\Sigma$ represented as a manifold obtained by gluing two copies of $D^{m-n+1}$ (we call such a homotopy sphere an {\it almost-sphere}) on the boundaries by a diffeomorphism as shown. Note also that such maps can characterize manifolds represented as total
 spaces of smooth bundles over $S^n$ with fibers diffeomorphic to $\Sigma$ with a condition on the structure of the map on the preimage of the target space with the interior of an $n$-dimensional standard closed disc in the innermost connected component of the regular value set removed. See the three papers by the author cited before.
\begin{figure}
\includegraphics[width=80mm]{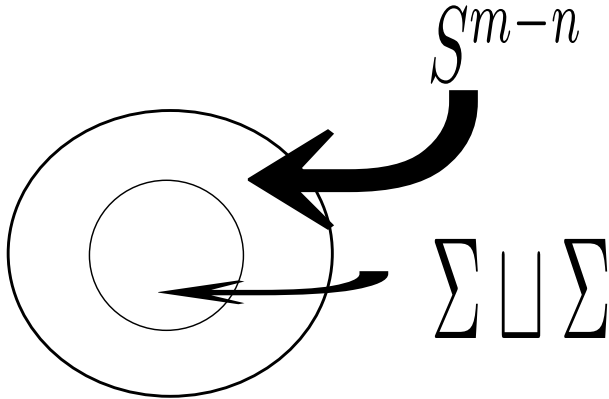}
\caption{The image of an explicit fold map into ${\mathbb{R}}^n$ ($n \geq 2$): the manifolds represent preimages of regular values and the circles represent the singular value set, diffeomorphic to the disjoint union of two copies of $S^{n-1}$.}
\label{fig:3}
\end{figure}
\end{enumerate}
\end{Ex} 

The first two examples are {\it special generic}: a {\it special generic} map is a fold map such that the index of each singular point is $0$. Studies of special generic maps and manifolds admitting these maps were started by Furuya and Porto, Saeki and Sakuma obtained various results and Nishioka and Wrazidlo recently obtained interesting results. Interesting respects of special generic maps such as strong restrictions on differentiable structures of manifolds admitting these maps will be also presented in Remark \ref{rem:1}. More concrete studies on special generic maps are demonstrated in \cite{furuyaporto}, \cite{saeki2}, \cite{saekisakuma} and also in \cite{nishioka} and \cite{wrazidlo}.

The author constructed explicit fold maps which may not be special generic such as ones in the last example just before, in \cite{kitazawa}, \cite{kitazawa2}, \cite{kitazawa4}, \cite{kitazawa5} and \cite{kitazawa6}, for example. In \cite{kobayashi} and \cite{kobayashi2}, Kobayashi also succeeded in such works independently, for example.

In this paper, we further study about construction in \cite{kitazawa6}. In the paper, the author constructed fold maps by
 surgery operations ({\it bubbling operations}) and investigated their {\it Reeb spaces}.

The {\it Reeb space} of a map between two smooth manifolds is defined as the space of all connected components
 of preimages. Reeb spaces inherit fundamental important invariants of the manifolds admitting the maps such as homology groups in suitable cases. Reeb spaces are fundamental and important tools in studying the manifolds in general. 
The author also investigated changes of homology groups of Reeb spaces by the operations. In this paper, we study about more precise algebraic invariants of Reeb spaces and the manifolds admitting the maps. More precisely, we study about cohomology rings.

The organization of the paper is as the following.

In section \ref{sec:2}, we review {\it Reeb spaces}, properties of Reeb spaces of special generic maps and {\it bubbling operations}; in \cite{kitazawa6}, we introduced such operations and we revise definitions a little in the present paper.
 We introduce explicit cases and for example, we observe that several simple examples including ones presented in FIGUREs \ref{fig:1}--\ref{fig:3} are obtained by finite iterations of such operations. 
 We also review results
 on the changes or differences of homology groups. In addition, we also introduce a result or Proposition \ref{prop:7}. According to this, for suitable fold maps such that preimages of regular values are disjoint unions of spheres, we can know algebraic invariants of the manifolds from the Reeb spaces. Such a proposition was first shown in \cite{saekisuzuoka} and later in \cite{kitazawa2} and \cite{kitazawa3}.

In section \ref{sec:3}, based on the theory and results of the previous section, as main works of the present paper, we investigate changes of cohomology rings of Reeb spaces. More precisely, as explicitly depicted in FIGURE \ref{fig:5} later, we change the topology of the Reeb space one after another and investigate the resulting topology. We can know structures of cohomology rings of the manifolds by virtue of Proposition \ref{prop:7} just before in cases where preimages of regular values are disjoint unions of spheres and where additional suitable differential topological conditions on the maps are assumed. In addition to the main results, Remark \ref{rem:1} also gives a comment on a new explicit relation between cohomology rings of Reeb spaces of fold maps of suitable classes studied in the present paper and cohomology rings of manifolds admitting these maps.

Section \ref{sec:4} is for the presentation of a general version of the main results obtained in the previous section.
 
  Throughout this paper, we assume that $M$ is a smooth, closed and connected manifold of dimension $m$, that
 $N$ is a smooth manifold of dimension $n$ with no boundary, that $f:M \rightarrow N$ is a smooth map and that the relation $m>n \geq 1$ holds.

In addition, the structure groups of bundles such that the fibers are (smooth) manifolds are assumed to be
 (subgroups of) the {\it diffeomorphism groups}: a {\it diffeomorphism group} of a manifold is a group consisting of all diffeomorphisms on the manifold.


\section{Reeb spaces, bubbling operations and fold maps such that preimages of regular values are disjoint unions of spheres.}
\label{sec:2}
\subsection{Definitions and fundamental properties of Reeb spaces and bubbling operations.}

 Let $X$ and $Y$ be topological spaces. For $p_1, p_2 \in X$ and for a continuous map $c:X \rightarrow Y$, 
 we define a relation ${\sim}_c$ on $X$ by the following rule: $p_1 {\sim}_c p_2$ if and only if $p_1$ and $p_2$ are in
 a same connected component of a preimage $c^{-1}(p)$. 
Thus ${\sim}_{c}$ is an equivalence relation on $X$. 
\begin{Def}
\label{def:2}
We denote the quotient space $X/{\sim}_c$ by $W_c$ and call $W_c$ the {\it Reeb space} of $c$.
\end{Def}

 We denote the induced quotient map from $X$ into $W_c$ by $q_c$. We can define $\bar{c}:W_c \rightarrow Y$ uniquely
 so that the relation $c=\bar{c} \circ q_c$ holds.

  For a (stable) fold map $c$, the Reeb space $W_c$ is regarded as a polyhedron: see also \cite{kobayashisaeki} for example. For example, for a
 Morse function and more generally, a smooth function on a closed manifold, the Reeb space
 is a graph and for a
 special generic map,
 the Reeb space is regarded as a smooth manifold immersed into the target manifold. We present this (see also section 2 of \cite{saeki}). We explain about fundamental stuffs on ({\it linear}) bundles.
 A {\it linear} bundle whose fiber is a standard closed (unit) disc or its boundary is a smooth bundle whose structure group acts on the fiber linearly. A linear bundle is {\it orientable} if the structure group is reduced to a rotation group.

 if an orientable linear bundle is {\it oriented}, then we can define the {\it Euler class} as a suitable cohomology class of the base space. We omit the definition of an {\it oriented} linear bundle. We review fundamental facts on linear bundles.

\begin{Prop}
\label{prop:2.0}
If an orientable linear bundle whose fiber is a standard sphere of dimension $k>0$ over a compact manifold $Y$ admits a section, then the Euler class is zero if it is oriented and the (co)homology group of the total space $X$ and that of $Y \times S^k$ are isomorphic for arbitrary coefficient commutative groups. Their cohomology rings are also isomorphic for arbitrary coefficient rings {\rm (}which are PID{\rm )}.  
\end{Prop} 

\begin{Prop}
\label{prop:2}
There exists a special generic map $f:M \rightarrow {\mathbb{R}}^n$ on an $m$-dimensional closed and compact manifold $M$ if and only if $M$ is obtained by gluing the following two manifolds by
 a bundle isomorphism between the $S^{m-n}$-bundles over the boundary $\partial P$ of a compact and connected manifold $P$, which appear naturally as a subbundle and a restriction of a bundle in the following explanation.
\begin{enumerate}
\item A smooth $S^{m-n}$-bundle over a compact smooth manifold $P$ satisfying $\partial P \neq \emptyset$ we can immerse into ${\mathbb{R}}^n$.
\item A linear $D^{m-n+1}$-bundle over $\partial P$.
\end{enumerate}

\end{Prop}

 Note that $P$ is regarded as the Reeb space of a special generic map on the manifold.

 Next, we review {\it bubbling operations}, introduced in \cite{kitazawa6}
\begin{Def}[\cite{kitazawa6}]
\label{def:3}
For a fold map $f:M \rightarrow N$, let $P$ be a connected component of the regular
 value set ${\mathbb{R}}^n-f(S(f))$. Let $S$ be a connected and orientable closed submanifold of
 $P$ with no boundary and $N(S)$, ${N(S)}_i$ and ${N(S)}_o$ be small closed tubular neighborhoods
 of $S$ in $P$ such that the relations ${N(S)}_i \subset {\rm Int} N(S)$ and $N(S) \subset {\rm Int} {N(S)}_o$
 hold. Furthermore, we can naturally regard ${N(S)}_o$ as a linear bundle whose fiber is
 an ($m-n+1$)-dimensional disc of radius $1$ and ${N(S)}_i$ and $N(S)$ are subbundles of the bundle ${N(S)}_o$ whose
 fibers are ($m-n+1$)-dimensional discs of radii $\frac{1}{3}$ and $\frac{2}{3}$, respectively. 
Furthermore, let the linear bundles be orientable
Let
 $f^{-1}({N(S)}_o)$ have a connected component $Q$ such that $f {\mid}_{Q}$
 makes $Q$ a bundle over ${N(S)}_o$.

Let us assume that we can obtain an $m$-dimensional closed manifold $M^{\prime}$ and
 a fold map $f^{\prime}:M^{\prime} \rightarrow {\mathbb{R}}^n$
 satisfying the following properties.
\begin{enumerate}
\item $M-{\rm Int} Q$ is a compact manifold with non-empty boundary and we can embed this into $M^{\prime}$ via a smooth embedding $e$.
\item $f {\mid}_{M-{\rm Int} Q}={f}^{\prime} {\mid}_{e(M-{\rm Int} Q)} \circ e$ holds.
\item ${f}^{\prime}(S({f}^{\prime}))$ is the disjoint union of $f(S(f))$ and $\partial N(S)$.
\item $(M^{\prime}-e(M-Q)) \bigcap {f^{\prime}}^{-1}({N(S)}_i)$ is empty or ${{f}^{\prime}} {\mid}_{(M^{\prime}-e(M-Q)) \bigcap {f^{\prime}}^{-1}({N(S)}_i)}$ makes $(M^{\prime}-e(M-Q)) \bigcap {f^{\prime}}^{-1}({N(S)}_i)$ a bundle over ${N(S)}_i$.
\end{enumerate}
These assumptions enable us to formulate the procedure of constructing $f^{\prime}$ from $f$ and we
 call it a {\it normal bubbling operation} to $f$ and, ${\bar{f}}^{-1}(S) \bigcap q_f(Q)$, which is homeomorphic
 to $S$, the {\it generating manifold} of the normal bubbling operation. 

Furthermore, let us suppose additional conditions.

\begin{enumerate}
\item
 ${{f}^{\prime}} {\mid}_{(M^{\prime}-e(M-Q)) \bigcap {f^{\prime}}^{-1}({N(S)}_i)}$ makes $(M^{\prime}-e(M-Q)) \bigcap {f^{\prime}}^{-1}({N(S)}_i)$ the disjoint union of two bundles over $N(S)$. In this case, the procedure is called a {\it normal M-bubbling operation} to $f$.
\item ${{f}^{\prime}} {\mid}_{(M^{\prime}-e(M-Q)) \bigcap {f^{\prime}}^{-1}({N(S)}_i)}$ makes $(M^{\prime}-e(M-Q)) \bigcap {f^{\prime}}^{-1}({N(S)}_i)$ the disjoint union of two bundles over $N(S)$ and the fiber of one of the bundles is an almost-sphere. In this case, the procedure is called a {\it normal S-bubbling operation} to $f$.
\end{enumerate}

In the definition above, let $S$ be the bouquet of finite connected and orientable closed submanifolds with no boundaries whose dimensions are smaller than $n$ of
 $P$ and $N(S)$, ${N(S)}_i$ and ${N(S)}_o$ be small regular neighborhoods
 of $S$ in $P$ such that a similar relation holds and that these three are isotopic as regular neighborhoods. Furthermore, let us assume that linear bundles obtained for each closed submanifold as in the cases of submanifolds are orientable
 By a similar way,
 we define a similar operation and call the operation a {\it bubbling operation} to $f$. We call $Q_0:={\bar{f}}^{-1}(S) \bigcap q_f(Q)$, which is homeomorphic
 to $S$, the {\it generating polyhedron} of the bubbling operation.

Last, if a fold map is not given and only a smooth manifold $N$ is given, then we can define
 a bubbling operation naturally. We thus obtain a special generic map $f:M \rightarrow N$ such that $f {\mid}_{S(f)}$ is an embedding. We call this a {\it default} bubbling operation.
\end{Def} 
In the following example, as in \cite{kitazawa6}, we present explicit and important facts on bubbling operations.

\begin{Ex}
\label{ex:2}
\begin{enumerate}
\item
\label{ex:2.1}
 A bubbling operation where the generating manifold is a point is a {\it bubbling surgery}, introduced in \cite{kobayashi2}, based on ideas of \cite{kobayashi}. \cite{kobayashisaeki} is closely related to such operations: as surgery operations, {\it R-operations} are defined as operations deforming stable maps from closed manifolds whose dimensions are larger than $2$ into the plane and preserving the topologies and the differentiable structures of the manifolds. 
Note that for example, the map presented in FIGURE \ref{fig:3} is obtained by a bubbling surgery after a default bubbling operation whose generating manifold is a point.
\item
\label{ex:2.2}
 By suitable default normal bubbling operations whose generating manifolds are points, we can obtain maps in FIGURE \ref{fig:1} (the dimensions of the manifolds are larger than those of target spaces). Moreover, for example, by a suitable default bubbling operation, we can obtain a special generic map on any closed manifold admitting a special generic map into the plane (see \cite{saeki}). FIGURE \ref{fig:2} shows examples of such maps. Note that the Reeb spaces of special generic maps into the plane on manifolds whose dimensions are larger than $2$ are manifolds represented as boundary connected sums of finite copies of $S^1 \times D^1$.

Let $m>n>1$.
Note that on a manifold represented as a connected sum of the products $S^{k_{j,1}} \times S^{k_{j,2}}$ satisfying the relations $k_{j,1}+k_{j,2}=m$, $0<k_{j,1} \leq n-1$ and $k_{j,2}>0$, by a default bubbling operation, we can construct a special generic map
 into ${\mathbb{R}}^n$ by taking a generating polyhedron as a bouquet of standard spheres in the family $\{S^{k_{j,1}}\}$.    
 
FIGURE \ref{fig:4} represents several simple default bubbling operations. 
The first figure accounts for an operation such that the target manifold is ${\mathbb{R}}$ and that the generating manifold is a point. The second figure accounts for an operation such that the target manifold is ${\mathbb{R}}^2$ and that the generating manifold is a circle and an operation such that the target manifold is ${\mathbb{R}}^2$ and that the generating polyhedron is a bouquet of a point and a circle (they are essentially same). The third figure accounts for an operation such that the target manifold is ${\mathbb{R}}^2$ and that the generating polyhedron is a bouquet of two circles. These figures can account for general
 bubbling operations such that preimages are not empty.

\begin{figure}
\includegraphics[width=80mm]{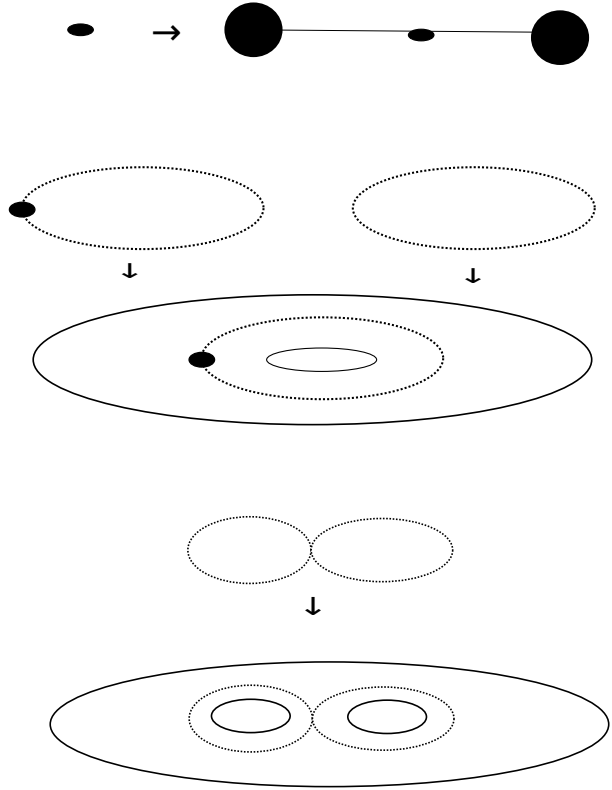}
\caption{Default bubbling operations and the images (singular value sets) of resulting special generic maps.}
\label{fig:4}
\end{figure}
\item
\label{ex:2.3}
For a fold map $f:M \rightarrow N$ and a connected component $P$ of the set $f(M)-f(S(f))$, let $S$ be a connected and orientable closed submanifold of
 $P$ such that there exists a connected component $S^{\prime}$ of $f^{-1}(S)$ and that $f {\mid}_{S^{\prime}}:S^{\prime} \rightarrow S$
 makes $S^{\prime}$ a trivial bundle over $S$. Let $f_{m,n,S}$ be a Morse function such that the following properties hold.

Let us define a function $f_{m,n,S}$ satisfying the following properties.
\begin{enumerate}
\item $f_{m,n,S}$ is a Morse function on a compact manifold one of connected components of whose boundary is the fiber
 $F$ of the bundle $S^{\prime}$ over $S$ and with exactly one singular point in its interior. At the singular point, the function does not have the maximal value or the minimal value.
\item The preimage of
 the maximal value is the connected component, diffeomorphic to $F$, and the preimage of the minimal value is the disjoint
 union of connected components of the boundary except the previous one if the disjoint union is not empty and is the singular point of $f_{m,n,S}$ if the disjoint union is empty (or the index of the singular point is $0$).  
\end{enumerate} 

 Then, by a bubbling operation to $f$ such that the generating manifold is $S^{\prime}$, we
 can obtain a new fold map $f^{\prime}:M^{\prime} \rightarrow {\mathbb{R}}^n$ satisfying the following properties
 where we abuse notation in Definition \ref{def:3}. We call this operation
 a {\it trivial} normal bubbling operation.   
\begin{enumerate}
\item $f {\mid}_{M-{\rm Int} Q}={f}^{\prime} {\mid}_{e(M-{\rm Int} Q)} \circ e$.
\item ${{f}^{\prime}} {\mid}_{{{f}^{\prime}}^{-1}({N(S)}_i) \bigcap (M^{\prime}-e(M-Q))}$ gives a trivial bundle over ${N(S)}_i$. 
\item There exists a connected component of ${f^{\prime}}^{-1}({N(S)}_o-{\rm Int}{N(S)}_i)$ such that the restriction map
 of ${f^{\prime}}$ to the component is for suitable coordinates represented as a product of the Morse
 function $f_{m,n,S}$ and ${\rm id}_{\partial {N(S)}_i}$.
\end{enumerate}
Moreover, let the normal bundle or tubular neighborhood of $S$ be regarded as a trivial bundle. ${N(S)}_o$ is
 represented by $S \times D^{n-\dim S}$ and $S$ is regarded
  as $S \times \{0\} \subset S \times D^{n-\dim S}$. For example, let $S$ be the standard sphere embedded as an unknot in the interior of an open ball in the interior
 of $P$. Furthermore, let the
   restriction of $f^{\prime}$ to ${f^{\prime}}^{-1}(N(S)_o)$ is, for suitable coordinates, represented as a product map of a surjective map $f^{\prime} {\mid}_{{f^{\prime}}^{-1}(D^{n-\dim S})}:{f^{\prime}}^{-1}(D^{n-\dim S}) \rightarrow D^{n-\dim S}$ where $D^{n-\dim S}$ is a fiber of the trivial bundle $N(S)_o$ and ${\rm id}_{S}$. Thus we call the previous operation a {\it strongly trivial} normal bubbling operation. Bubbling surgeries, presented
 in Example \ref{ex:2} (\ref{ex:2.1}), are strongly trivial normal bubbling operations, for example. Last, we can extend the notion of a trivial normal bubbling operation to cases of bubbling operations. 
\item
\label{ex:2.4}
In the previous example, if $F_1$ and $F_2$ are closed and connected manifolds
 such that the manifold $F$ is represented as a connected sum of $F_1$ and $F_2$, then, we can consider $f_{m,n.S}$ so that the boundary of
 its source manifold consists of three connected components and the boundary with the connected
 component $F$ removed is the disjoint union of $F_1$ and $F_2$. In this case, the bubbling operation is an M-bubbling operation.
 We can take $F_1$ as any almost-sphere of dimension $m-n$ and $F_2$ suitably and in this case, the operation is an S-bubbling operation.
FIGURE \ref{fig:5} later accounts for an explicit M-bubbling operation and the change of the topology of the Reeb space.
\item
\label{ex:2.5}
If we perform an S-bubbling operation to a fold map such that the preimages of regular values are always disjoint unions of almost-spheres (standard spheres), then we obtain a fold map satisfying the same property. A {\it simple} fold map $f$ is a fold map such that the map ${q_f} {\mid}_{S(f)}:S(f) \subset M \rightarrow W_f$ is
 injective. These maps were systematically studied in \cite{sakuma} for example. A special generic map and a fold map $f$ such that the map $f {\mid}_{S(f)}$ is an embedding are simple fold maps.
 If we perform an M-bubbling operation to a simple fold map (such that the restriction map obtained by the restriction to the singular set is embedding), then the resulting fold map is also a map satisfying this. 
\end{enumerate}
\end{Ex}

The following lemma is a key lemma in section \ref{sec:3} and it follows immediately by the definition of an M-bubbling operation. 
\begin{Lem}
\label{lem:1}
Let $f$ be a fold map. If an M-bubbling operation is performed to $f$ and a new map $f^{\prime}$ is obtained, then $W_f$ is a proper subset of $W_{{f}^{\prime}}$ such that for the map $\bar{{f}^{\prime}}:W_{f^{\prime}} \rightarrow N$, the restriction to $W_f$ coincides with $\bar{f}:W_f \rightarrow N$.
\end{Lem}

FIGURE \ref{fig:5} represents an example for Lemma \ref{lem:1} where $n=2$ holds.
\begin{figure}
\includegraphics[width=60mm]{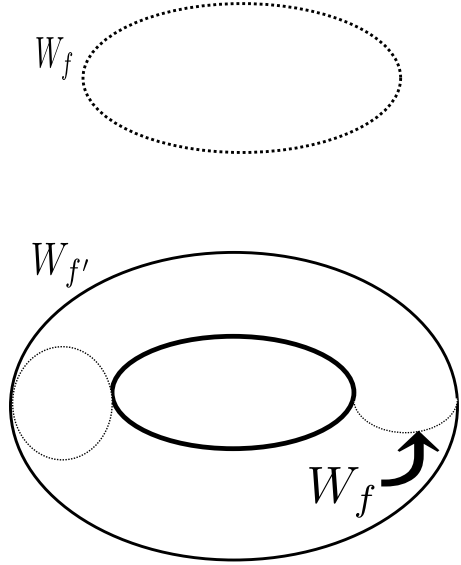}
\caption{Lemma \ref{lem:1} in a case of $n=2$.}
\label{fig:5}
\end{figure}
Propositions \ref{prop:3}-\ref{prop:7.0} in this section are fundamental and key tools in the present paper. Hereafter, we denote a commutative ring by $R$ and if it is a principle ideal domain, then we abbreviate "principle ideal domain" as {\it PID}. For coefficient commutative groups of homology and cohomology groups, in most cases of the present paper, we take commutative rings as PIDs. 
\begin{Prop}
\label{prop:3}
Let $R$ be a PID. 
Let $f:M \rightarrow N$ be a fold map. Let $f^{\prime}:{M}^{\prime} \rightarrow N$ be a fold map obtained by an
 M-bubbling operation to $f$. Let $S$ be the generating polyhedron of the M-bubbling
 operation. Let $k$ be a positive integer and $S$ be represented as the bouquet of submanifolds $S_j$ where $j$ is an integer satisfying $1 \leq j \leq k$.  
In this situation, for any integer $0 \leq i<n$, there exists an isomorphism of modules represented by

$$H_{i}(W_{{f}^{\prime}};R) \cong H_{i}(W_f;R) \oplus {\oplus}_{j=1}^{k} (H_{i-(n-{\dim} S_j)}(S_j;R))$$

and there exists an isomorphism of modules represented by $H_{n}(W_{{f}^{\prime}};R) \cong H_{n}(W_f;R) \oplus R$.
\end{Prop}
A more rigorous proof with explanations on Mayer-Vietoris sequences, homology groups of product
 bundles, and so on, is presented \cite{kitazawa6} and we present a shorter proof here. We also note that in discussions
 later, for example ones in section \ref{sec:3} including the proofs of results of the present paper, such precise explanations
 on algebraic topological methods are omitted.

\begin{proof}
For each $S_j$, we can take a small closed tubular neighborhood, regarded as
a total space of a linear $D^{n-\dim S_j}$-bundle over $S_j$. By the definition of an M-bubbling operation, we can
 see that a small regular neighborhood of $S$ is represented as a boundary connected sum of the
 closed tubular neighborhoods. We may consider that $W_{f^{\prime}}$ is obtained by the
 attaching a manifold represented as a connected sum of total spaces
 of linear $S^{n-\dim S_j}$-bundles over $S_j$ ($1 \leq j \leq k$) by considering the fiber, diffeomorphic to $D^{n-\dim S_j}$, before as a hemisphere of the fiber, diffeomorphic to $S^{n-\dim S_j}$, and identifying the subspace obtained by restricting the
 space to fibers $D^{m-\dim S_j}$ with the original regular neighborhood. Lemma \ref{lem:1} and FIGURE \ref{fig:5} may help us to understand the topology of the resulting space $W_
{f^{\prime}} \supset W_f$. The manifold represented as a connected sum of total spaces of linear $S^{n-\dim S_j}$-bundles
 over $S_j$ ($1 \leq j \leq k$) can be regarded as products in this situation, since we investigate the homology groups and the bundles admit sections (Proposition \ref{prop:2.0}). The images of the sections are regarded as the submanifolds $S_j$ and are obtained by taking the origin
 in each fiber $D^{n-\dim S_j} \subset S^{n-\dim S_j}$. This observation on the
 topologies of $W_f$ and $W_{f^{\prime}}$ yields the result.
\end{proof}

The following proposition has been shown in \cite{kitazawa6}. We can show this by applying Proposition \ref{prop:3} one
 after another. More explicitly, we take a suitable family of generating polyhedra, which are bouquets of finite
 numbers of standard spheres.
\begin{Prop}
\label{prop:4}
Let $R$ be a PID. For any integer $0 \leq j \leq n$, we define $G_j$ as a free and finitely generated
 module over $R$ so that $G_0$ is trivial and that $G_n$ is not zero. In this situation, by
 a finite iteration of normal M-bubbling operations to a map $f$, we obtain a fold map $f^{\prime}$ such
  that the group $H_j(W_{{f}^{\prime}};R)$ is isomorphic to the group $H_j(W_f;R) \oplus G_j$. 
\end{Prop}

We restrict M-bubbling operations in Propositions \ref{prop:3} and \ref{prop:4} to normal ones and thus we have the following.

\begin{Prop}
\label{prop:5}
Let $R$ be a PID. 
For a fold map $f:M \rightarrow N$, let $f^{\prime}:{M}^{\prime} \rightarrow N$ be a fold map obtained by a
 normal M-bubbling operation to $f$ and let $S$ be the generating manifold of the normal
 M-bubbling operation and of dimension $k<n$. Then for any integer $i$, there exists an isomorphism of modules represented by $H_i(W_{{f}^{\prime}};R) \cong H_{i}(W_f;R) \oplus (H_{i-(n-k)}(S;R))$.
\end{Prop}
\begin{Prop}
\label{prop:6}
For any integer $0 \leq j \leq n$, we define $G_j$ as a free finitely generated module over a PID $R$ so
  that $G_0$ is a trivial module and $G_n$ is not
 a trivial module. Let the sum ${\sum}_{j=1}^{n-1} {\rm rank} \quad G_j$ of the ranks
 of $G_j$ is not larger than the rank of $G_n$. In this situation, by a finite iteration
 of normal M-bubbling {\rm (}S-bubbling{\rm )} operations starting from $f$, we obtain a fold map $f^{\prime}$ and two modules $H_j(W_{{f}^{\prime}};R)$ and $H_j(W_f;R) \oplus G_j$ are isomorphic. 
\end{Prop}
 In the proof of Proposition \ref{prop:3}, let us construct an isomorphism yielding the isomorphism of modules represented by

 $$H_{i}(W_{{f}^{\prime}};R) \cong H_{i}(W_f;R) \oplus {\oplus}_{j=1}^{k} (H_{i-(n-{\dim} S_j)}(S_j;R)).$$

Lemma \ref{lem:1} produces an inclusion $i_{(f,f^{\prime}),S}:W_f \rightarrow W_{f^{\prime}}$. We may regard $S_j \subset W_f$. Consider the ($n-\dim S_j$)-cycle represented by a fiber $S^{n-\dim S_j} \supset D^{n-\dim S_j}$ of the bundle over $S_j$ appearing in the proof. Consider an ($i-(n-{\dim} S_j)$)-cycle representing a class $c \in H_{i-(n-{\dim} S_j)}(S_j;R)$ and the class represented by the fiber and as a result naturally we obtain a class $c^{\prime} \in H_{i}(W_{{f}^{\prime}};R)$ (topologically we obtain an object like a trivial bundle or we use a kind of so-called {\it prism} operators). A desired isomorphism is given as the direct sum of ${i_{(f,f^{\prime}),S}}_{\ast}$ and a monomorphism ${\phi}_{(f,f^{\prime}),S}(c):=c^{\prime}$.

\begin{Def}
\label{def:4}
We call the monomorphisms ${i_{(f,f^{\prime}),S}}_{\ast}$ and ${\phi}_{(f,f^{\prime}),S}(c):=c^{\prime}$, an {\it inclusion morphism} and a {\it bubbling morphism} of the M-bubbling operation, respectively. The direct sum of these two morphisms is called the {\it canonical homology isomorphism} of the operation.   
\end{Def}


\begin{Prop}
\label{prop:7.0}
Let $R$ be a PID.
Let $f:M \rightarrow N$ be a fold map. Let $f^{\prime}:{M}^{\prime} \rightarrow N$ be a fold map obtained by an
 M-bubbling operation to $f$. Let $S$ be the generating polyhedron of the M-bubbling
 operation. Let $k$ be a positive integer and $S$ be represented as the bouquet of submanifolds $S_j$ where $j$ is an integer satisfying $1 \leq j \leq k$.  
Then, for any integer $0\leq i<n$, there exists an isomorphism of modules represented by

$$H^{i}(W_{{f}^{\prime}};R) \cong H^{i}(W_f;R) \oplus {\oplus}_{j=1}^{k} (H^{i-(n-{\dim} S_j)}(S_j;R))$$

and there exists an isomorphism of groups represented by $H^{n}(W_{{f}^{\prime}};R) \cong H^{n}(W_f;R) \oplus R$.
\end{Prop}

This can be shown similarly to Proposition \ref{prop:3}.
 In the proof of Proposition \ref{prop:3}, let us construct an isomorphism yielding the relation

 $$H^{i}(W_{{f}^{\prime}};R) \cong H^{i}(W_f;R) \oplus {\oplus}_{j=1}^{k} (H^{i-(n-{\dim} S_j)}(S_j;R)).$$

For each $i$-cocycle representing an element of $H^{i}(W_f;R)$, we can define a cocycle
representing an element of $H^{i}(W_{{f}^{\prime}};R)$
 in a canonical way as the following. 
Considering a suitable triangulation of $W_f$ and a suitable one of $W_{f^{\prime}} \supset W_f$ regarded as an extension of the triangulation of $W_f$. 
\begin{enumerate}
\item At each $i$-chain in the newly attached space satisfying the following property, the value is $0$: for the chain, if the coefficient at an $i$-simplex is not zero, then the simplex is not contained in the original Reeb space $W_f$. 
\item At each $i$-chain regarded as a chain contained in the original Reeb space $W_f$, the value is same as the value of the original cocycle at the same chain.
\end{enumerate}
For each $c \in H^{i}(W_f;R)$ and a cocycle representing this, the new class ${i_{(f,f^{\prime}),S}}^{{\ast}^{\prime}}(c) \in H^{i}(W_{{f}^{\prime}};R)$ is well-defined as a class represented by a new cocycle satisfying the properties. ${i_{(f,f^{\prime}),S}}^{{\ast}^{\prime}}$ is defined as a monomorphism from $H^{i}(W_{f};R)$ to $H^{i}(W_{{f}^{\prime}};R)$ mapping $c$ to ${i_{(f,f^{\prime}),S}}^{{\ast}^{\prime}}(c)$.

Consider an ($n-\dim S_j$)-cocycle representing a generator of the module $H^{n-\dim S_j}(S^{n-\dim S_j};R)$ of the fiber $S^{n-\dim S_j} \supset D^{n-\dim S_j}$ of the bundle over $S_j$ appearing in the proof of Proposition \ref{prop:3}. Consider an ($i-(n-{\dim} S_j)$)-cocycle representing a class $c \in H^{i-(n-{\dim} S_j)}(S_j;R)$ and a cocycle representing a generator of the module $H^{n-\dim S_j}(S^{n-\dim S_j};R)$ and as a result naturally we obtain an $i$-cocycle of $W_{f^{\prime}}$ representing a class $c^{\prime} \in H^{i}(W_{{f}^{\prime}};R)$. A desired isomorphism is given as the direct sum of ${i_{(f,f^{\prime}),S}}^{\ast}$ and a monomorphism given by ${\phi}_{(f,f^{\prime}),S}(c):=c^{\prime}$.

\begin{Def}
\label{def:5}
We call the monomorphisms ${i_{(f,f^{\prime}),S}}^{{\ast}^{\prime}}$ and ${\phi}_{(f,f^{\prime}),S}(c):=c^{\prime}$ an {\it inclusion morphism} and a {\it bubbling morphism} of the M-bubbling operation, respectively. The direct sum of these two homomorphisms is called the {\it canonical cohomology isomorphism} of the operation. 
\end{Def}
We have the following two propositions immediately
\begin{Prop}
We can define the inclusion morphism ${i_{(f,f^{\prime}),S}}^{{\ast}^{\prime}}$ of the M-bubbling operation as a monomorphism which is not only regarded as a homomorphism between the underlying graded modules but also a homomorphism between the underlying graded algebras.  
\end{Prop}
\begin{Prop}
We can extend notions, terminologies, and so on, for cases of a single M-bubbling operation to cases of finite iterations of M-bubbling operations.
\end{Prop}
\begin{Def}
\label{def:6}
For cases of finite iterations of M-bubbling operations, as for cases of single M-bubbling operations, we abuse terminologies in Definitions \ref{def:4} and \ref{def:5}.
\end{Def}
\subsection{Fold maps such that preimages of regular values are disjoint unions of spheres.}
In the end of this section, we review a proposition for fold maps such that preimages of regular values are disjoint unions of spheres. These maps are important and appear on various situations: special generic maps satisfy the property and a map in Example \ref{ex:1} (\ref{ex:1.3}) does. 
The following is a proposition for simple fold maps, appearing in Example \ref{ex:2} (\ref{ex:2.5}) and so on, and this is a key proposition to know algebraic invariants of source manifolds from Reeb spaces under appropriate constraints. Several statements such as one on an isomorphism between graded commutative algebras obtained by replacing modules of higher degrees of the original cohomology rings by $\{0\}$ were not shown in the original articles. However, we can show this in a manner similar to the used manners: the key is fundamental theory of handle decompositions of (PL or smooth) manifolds. 

\begin{Prop}[\cite{saekisuzuoka}  (\cite{kitazawa2})]
\label{prop:7}
Let $A$ be a commutative group. 
Let $m$ and $n$ be integers satisfying $m>n \geq 1$. Let $M$ be a closed
 and connected orientable manifold of dimension $m$ and $N$ be an $n$-dimensional manifold with no boundary.
  
Then, for a simple fold map $f:M \rightarrow N$ such that preimages of regular values
 are always disjoint unions of almost-spheres and that indices of singular points are always $0$ or $1$, two induced homomorphisms
 ${q_f}_{\ast}:{\pi}_j(M) \rightarrow {\pi}_j(W_f)$, ${q_f}_{\ast}:H_j(M;A) \rightarrow H_j(W_f;A)$, and ${q_f}^{\ast}:H^j(W_f;A) \rightarrow H^j(M;A)$ are isomorphisms of groups and modules for $0 \leq j \leq m-n-1$. 

Furthermore, we have the following for specific cases.
\begin{enumerate}
\item Let $A$ be a commutative ring. Let $J$ be the set of integers smaller than or equal to $0$ and larger than or equal to $m-n-1$ and let ${\oplus}_{j \in J} H^{j}(W_f;A)$ and ${\oplus}_{j \in J} H^{j}(M;A)$ be the algebras obtained by replacing the $j$-th modules of the cohomology rings $H^{\ast}(W_f;A)$ and $H^{\ast}(M;A)$ by $\{0\}$ for $j \geq m-n$, respectively. In this situation, $q_f$ induces an isomorphism between the commutative algebras
${\oplus}_{j \in J} H^j(W_f;A)$ and ${\oplus}_{j \in J} H^{j}(M;A)$ via the restriction of ${q_f}^{\ast}$. 
\item Furthermore, if $A$ is a PID and $m=2n$ holds, then the rank of $M$ is twice the rank of $W_f$ and in addition if $H_{n-1}(W_f;A)$, which is isomorphic to $H_{n-1}(M;A)$, is free, then these two modules are also free.
\end{enumerate} 
\end{Prop}

\section{On cohomology rings of Reeb spaces of fold maps and manifolds admitting the maps.}
\label{sec:3}
We investigate not only homology groups, but also cohomology rings of the resulting Reeb spaces
 and present the results as main results. In this section, we treat graded commutative algebras over fixed commutative rings including cohomology rings.
For a graded commutative algebra over a commutative ring, we define the {\it $i$-th module} of this as the module consisting of all elements of degree $i$ and the $0$-th module is assumed to be the commutative ring forgetting the structure of the ring. 
For a module $A$ over a commutative ring $R$. Let there exist a unique identity element $1 \in R$ satisfying $1 \neq 0 \in R$ and an element $a \in A$ satisfying the following properties.
\begin{enumerate}
\item $a$ is not represented as $ra^{\prime}$ for any element $(r,a^{\prime}) \in R \times A$ where $r$ is not a unit.
\item There exists a submodule $B$ of $A$ such that $A$ is represented as the internal direct sum of the submodule generated by the one element set $\{a\}$ and this. 
\end{enumerate}
We can define a homomorphism over $R$, $a^{\ast}:A \rightarrow R$ in a unique way so that the following properties hold. $a^{\ast}$ is said to be the {\it dual} of $a$.
\begin{enumerate}
\item $a^{\ast}(a)=1$.
\item For any submodule $B$ as before $a^{\ast}(B)=0$.
\end{enumerate}
For suitable homology class we can define the dual as a cohomology class and we apply this in the present paper.
\subsection{A connected sum of two smooth maps whose codimensions are negative.}
First we introduce a {\it connected sum} of two smooth maps whose codimensions are negative. This is also a fundamental operation in constructing maps. See \cite{saeki} and see also \cite{kobayashisaeki}, in which such operations were used to construct new maps from given pairs of special generic maps into fixed Euclidean spaces and generic maps on closed manifolds of dimensions larger than $2$ into the plane.

Let ${\pi}_{m+1,n}:{\mathbb{R}}^{m+1} \rightarrow {\mathbb{R}}^n$ defined by
$${\pi}_{m+1,n}((x_1,\cdots,x_{m+1})):=(x_1,\cdots,x_n)$$ be the canonical 
projection. Set ${{\mathbb{R}}^n}^+:=\{x=(x_1.\cdots,x_n) \in {\mathbb{R}}^n \mid x_1 \geq 0 \}$. 
The restriction of the map ${\pi}_{m+1,n}$ to the unit sphere $S^m$ is the canonical projection as presented in Example \ref{ex:1} (\ref{ex:1.1}) and we denote its restriction to the preimage of ${{\mathbb{R}}^n}^+$ by ${\pi}_{{S^m}^{+},n}$: its domain is diffeomorphic to $D^{m}$ (see FIGURE \ref{fig:6}). 

\begin{figure}
\includegraphics[width=40mm]{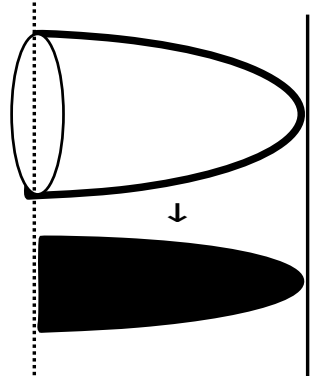}
\caption{${\pi}_{{S^m}^{+},n}$ (an arrow is for the projection).}
\label{fig:6}
\end{figure}

Let $m>n \geq 1$ be integers, $M_i$ ($i=1,2$) be a closed and connected manifold of dimension $m$ and $f_1:M_1 \rightarrow {{\mathbb{R}}}^n$
and $f_2:M_2 \rightarrow {\mathbb{R}}^n$ be smooth maps.
Let $P_i$ ($i=1,2$) be the closure of a region obtained by
 a hyperplane in ${\mathbb{R}}^n$ such that
for the map ${f_i} {\mid}_{{f_i}^{-1}(P_i)}:{f_i}^{-1}(P_i) \rightarrow P_i$, there exist diffeomorphisms $\Phi$ and $\phi$
 satisfying the relation

$$\phi \circ {f_i} {\mid}_{{f_i}^{-1}(P_i)}={\pi}_{{S^m}^{+},n} \circ \Phi$$

and we can glue the maps ${f_i} {\mid}_{{f_i}^{-1}({\mathbb{R}}^n-{\rm Int} P_i)}:{f_i}^{-1}({\mathbb{R}}^n-{\rm Int} P_i) \rightarrow {\mathbb{R}}^n-{\rm Int} P_i$ ($i=1,2$) by suitable diffeomorphisms on the boundaries to
 obtain a new smooth map into ${\mathbb{R}}^n$ so that the resulting manifold is represented as a connected sum of the original manifolds $M_1$ and $M_2$. 
The resulting map is said to be a {\it connected sum} of $f_1$ and $f_2$. See also FIGURE \ref{fig:7}.

For the two maps $f_1$ and $f_2$, if for suitable diffeomorphisms ${\phi}_1$ and ${\phi}_2$ on ${\mathbb{R}}^n$, we can define a {\it connected sum} of ${\phi}_1 \circ f_1$ and ${\phi}_2 \circ f_2$ similarly, then the resulting map is also called a connected sum of $f_1$ and $f_2$
\begin{figure}
\includegraphics[width=60mm]{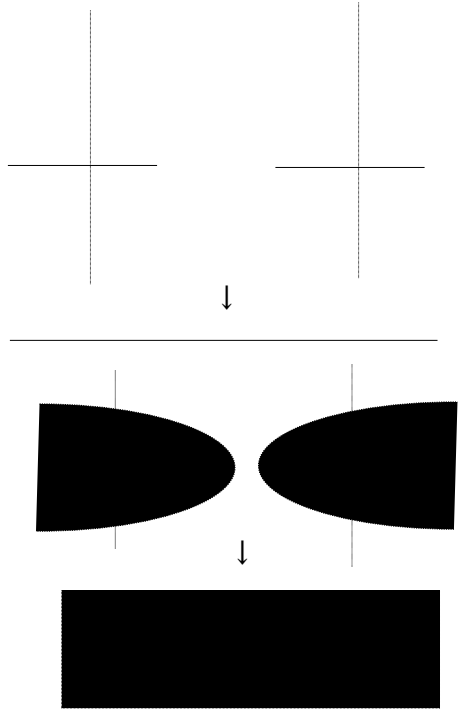}
\caption{Presentations of connected sums of two smooth maps for cases $n=1,2$ (the upper figure is for $n=1$ and the lower figure is for $n=2$) via the images of the maps.}
\label{fig:7}
\end{figure}

\begin{Ex}
\label{ex:3}
Most of maps presented in Example \ref{ex:2} (\ref{ex:2.2}) (and FIGURE \ref{fig:2}) are obtained by finite iterations of
 connected sums of maps obtained by strongly trivial default normal bubbling operations.
\end{Ex}

\subsection{Results.}
Related to the explanation of section \ref{sec:2}, we can show a result on cohomology rings of Reeb spaces. Moreover, if we can apply Proposition \ref{prop:7}, then we obtain a result on those of the resulting manifolds.

We have the following.
\begin{Prop}
\label{prop:8}
Let $l_1$,$l_2$ and $n$ be positive integers satisfying $n \geq 2$.
Let $Q$ be a compact $n$-dimensional manifold represented as a boundary connected sum of $l_1$ manifolds each of
 which is represented as the product of a standard sphere and a standard closed disc. Let $j$ be an integer satisfying $1 \leq j \leq l_1$ and $l_{1,j}$ be an integer satisfying $1 \leq l_{1.j} \leq n-1$. We represent each of the $l_1$ manifolds by $S_{D.j}:=S^{l_{1,j}} \times D^{n-l_{1,j}}$.
Let $P$ be a compact polyhedron represented as a bouquet of $l_2$ standard spheres
 the dimension of each of which is greater than or equal to $1$ and smaller than or equal to $n-2$. Let $j$ be an integer satisfying the relation $1 \leq j \leq l_2$ and $l_{2,j}$ be an integer satisfying the relation $1 \leq l_{2,j} \leq n-2$ and let us assume that $j$-th sphere of the polyhedron is diffeomorphic to $S^{l_{2,j}}$. 

Let $Q_j$ be the set of all integers
 $1 \leq {j}^{\prime} \leq l_1$ satisfying $l_{2,j}=l_{1,{j}^{\prime}}$. We take an arbitrary integer as $n_{j,{j}^{\prime}}$ for each $j^{\prime} \in Q_j$.
 
We can realize $P$ as a subpolyhedron in $Q$ satisfying the following properties: we denote this by $S$ and each of the $l_2$ spheres by $S_j$, which is diffeomorphic to $S^{l_{2,j}}$.
\begin{enumerate}
\item $S$ is in the interior of a collar neighborhood of $\partial Q$ in $Q$ and each of $l_2$ spheres $S_j$ is regarded as a closed submanifold of $Q$.
\item We can set an $l_{2,j}$-cycle {\rm (}$l_{1,j^{\prime}}$-cycle{\rm )} representing a class ${\nu}_{{j}^{\prime}} \in H_{l_{2,j}}(Q;\mathbb{Z})$ and represented by the canonically defined sphere $S^{l_{1,j^{\prime}}} \times \{0\} \subset S^{l_{1,j^{\prime}}} \times {\rm Int} D^{n-l_{1,{j}^{\prime}}} \subset S^{l_{1,j^{\prime}}} \times D^{n-l_{1,{j}^{\prime}}} = S_{D,j} \subset Q$. The class represented by $S_j$ is represented as ${\Sigma}_{j^{\prime} \in Q_j} n_{j,{j}^{\prime}}{\nu}_{{j}^{\prime}}$.
\end{enumerate}
\end{Prop}

We omit a rigorous proof and we only explain some key ingredients. 

In FIGURE \ref{fig:8}, we explicitly present a case where $n=3$ holds, In this case, $l_1=2$ holds (the manifold $Q$ is represented as a boundary connected sum of $S^2 \times D^1$ and $S^1 \times D^2$) and
$l_2=3$ holds: the polyhedron $P$ ($S$) is a bouquet of a circle and two copies of the $2$-dimensional standard sphere: note that the class represented by the circle is $p$ times the class represented by $S^1 \times \{0\} \subset {\rm Int} Q$ where $p$ is an arbitrary integer and that the classes represented by the $2$-dimensional spheres are zero and the class represented by $S^2 \times \{0\} \subset S^2 \times {\rm Int} D^1 \subset S^2 \times D^1 \subset Q$, respectively. Based on this explicit observation, we can consider a general argument and give a proof.

Note also that the last part or the coefficients is based on Whitney's theory on embeddings.   

\begin{figure}
\includegraphics[width=60mm]{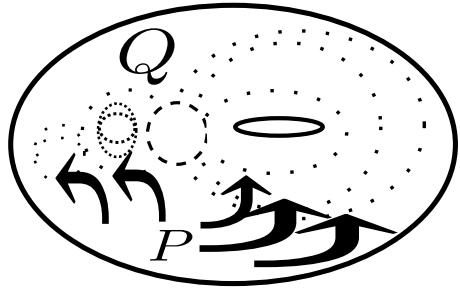}
\caption{A case where for the dimension, $n=3$ holds: the three arrows in the right-hand side point the circle in the polyhedron $S$ (, denoted by $P$,) and the other arrows point the $2$-dimensional spheres in $S$, respectively, and the thick dotted circle in the center is for abbreviating the circles and spheres before partially.}
\label{fig:8}
\end{figure}

For a finite set $X$, we denote the cardinality of $X$ by $\sharp X$. Based on Proposition \ref{prop:8}, we present a proposition.

For a stable special generic map whose image is the $n$-dimensional manifold in Proposition \ref{prop:8} such that the restriction to the singular set is an embedding, we have the following proposition where abuse the notation and the terminologies. 

\begin{Prop}
\label{prop:9}
Let $m>n \geq 2$ be integers.
Let $f$ be a fold map on an $m$-dimensional closed and connected manifold $M$ into ${\mathbb{R}}^n$ obtained by a finite iteration of bubbling operations starting from a fold map such that the restriction map to the set of all singular points of index $0$ is an embedding, that the image and the Reeb space are a compact $n$-dimensional manifold represented as a boundary connected sum of $l_1$ manifolds each of which is represented as the product of a standard sphere and a standard closed disc and diffeomorphic to the $n$-dimensional manifold of Proposition \ref{prop:8} and that the image of the restriction map to the set of singular points of indices $0$ before is the boundary. Let $j$ be an integer satisfying the relation $1 \leq j \leq l_1$ and $l_{1,j}$ be an integer satisfying the relation $1 \leq l_{1,j} \leq n-1$ and we represent each of the $l_1$ manifolds by $S_{D.j}:=S^{l_{1,j}} \times D^{n-l_{1,j}}$ as Proposition \ref{prop:8}.

If a fold map ${f^{\prime}}$ is obtained by an M-bubbling operation whose generating polyhedron is $S$ PL homeomorphic to $P$ in the explanation of the result of Proposition \ref{prop:8} to $f$, then we have the following statements where we abuse notation in Proposition \ref{prop:8} together with identifications of Reeb spaces as subpolyhedra of Reeb spaces obtained as the results of M-bubbling operations to the original maps as Lemma \ref{lem:1} and where $R$ is a PID having a unique identity element $1 \neq 0 \in R$.
\begin{enumerate}
\item
\label{prop:9.1}
 There exist isomorphisms of modules represented by $H_k(W_{{f}^{\prime}};R) \cong H_k(W_{f};R) \oplus R^{\sharp {\{1 \leq j \leq l_2 \mid l_{2,j}=n-k\}}}$ and $H^k(W_{{f}^{\prime}};R) \cong H^k(W_{f};R) \oplus R^{\sharp {\{1 \leq j \leq l_2 \mid l_{2,j}=n-k\}}}$ for $2 \leq k \leq n-1$. Moreover, there exist isomorphisms of modules represented by $H_1(W_{{f}^{\prime}};R) \cong H_1(W_{f};R)$, $H^1(W_{{f}^{\prime}};R) \cong H^1(W_{f};R)$, $H_n(W_{{f}^{\prime}};R) \cong H_n(W_{f};R) \oplus R$ and $H^n(W_{{f}^{\prime}};R) \cong H^n(W_{f};R) \oplus R$. 
Moreover, the isomorphisms $H_k(W_{{f}^{\prime}};R) \cong H_k(W_{f};R) \oplus R^{\sharp {\{1 \leq j \leq l_2 \mid l_{2,j}=n-k\}}}$ for $2 \leq k \leq n-1$, $H_1(W_{{f}^{\prime}};R) \cong H_1(W_{f};R)$ and $H_n(W_{{f}^{\prime}};R) \cong H_n(W_{f};R) \oplus R$ are given by canonical homology isomorphisms of the bubbling operation and for cohomology groups, the isomorphisms are also given by canonical cohomology isomorphisms of the bubbling operation.
\item
\label{prop:9.2}
 We can take a class ${{\nu}_j}^{\ast} \in H^{l_{1,j}}(W_f;R)$ such that ${{\nu}_j}^{\ast}({\nu}_j)=1$ satisfying the following.
\begin{enumerate}
\item We can take a submodule $C \subset H^{l_{1,j}}(W_f;R)$ such that the internal direct sum of the submodule generated by the set $\{{\nu}_j\}$ and $C$ is $H^{l_{1,j}}(W_f;R)$.
\item ${{\nu}_j}^{\ast}(C)=0$.
\end{enumerate} 
 For any pair $({{\nu}_{j_1}}^{\ast} \in H^{l_{1,j_1}}(W_f;R),{{\nu}_{j_2}}^{\ast} \in H^{l_{1,j_2}}(W_f;R))$, the product of ${i_{(f,f^{\prime}),S}}^{{\ast}^{\prime}}({{\nu}_{j_1}}^{\ast})$ and ${i_{(f,f^{\prime}),S}}^{{\ast}^{\prime}}({{\nu}_{j_2}}^{\ast})$ in $H^{\ast}(W_{{f}^{\prime}};R)$ vanishes.
\item
\label{prop:9.3.0}
 For relations $H_k(W_{{f}^{\prime}};R) \cong H_k(W_{f};R) \oplus R^{\sharp {\{1 \leq j \leq l_2 \mid l_{2,j}=n-k\}}}$ and $H^k(W_{{f}^{\prime}};R) \cong H^k(W_{f};R) \oplus R^{\sharp {\{1 \leq j \leq l_2 \mid l_{2,j}=n-k\}}}$ for $2 \leq k \leq n-1$, there exists an isomorphism ${\phi}_{(f,f^{\prime}),S,R,k}$ from the summand $R^{\sharp {\{1 \leq j \leq l \mid l_{2,j}=n-k\}}}$ onto the domain of ${\phi}_{(f,f^{\prime}),S}$. For relations $H_n(W_{{f}^{\prime}};R) \cong H_n(W_{f};R) \oplus R$ and $H^n(W_{{f}^{\prime}};R) \cong H^n(W_{f};R) \oplus R$, similar statements including the fact that a similar isomorphism ${\phi}_{(f,f^{\prime}),S,R,n}$ is defined hold.
\item
\label{prop:9.3}
For any element $a_1$ of the summand $R^{\sharp {\{1 \leq j \leq l_2 \mid l_{2,j}=n-k_1\}}}$ of $$H^{k_1}(W_{{f}^{\prime}};R) \cong H^{k_1}(W_{f};R) \oplus R^{\sharp {\{1 \leq j \leq l_2 \mid l_{2,j}=n-k_1\}}}$$ and for any element $a_2$ of the summand $R^{\sharp {\{1 \leq j \leq l_2 \mid l_{2,j}=n-k_2\}}}$ in $$H^{k_2}(W_{{f}^{\prime}};R) \cong H^{k_2}(W_{f};R) \oplus R^{\sharp {\{1 \leq j \leq l_2 \mid l_{2,j}=n-k_2\}}}$$ in {\rm (\ref{prop:9.1})}, the product of ${\phi}_{(f,f^{\prime}),S} \circ {\phi}_{(f,f^{\prime}),S,R,k_1}(a_1)$ and ${\phi}_{(f,f^{\prime}),S} \circ {\phi}_{(f,f^{\prime}),S,R,k_2}(a_2)$ in $H^{\ast}(W_{{f}^{\prime}};R)$ vanishes.
\item
\label{prop:9.4}
For any class ${{\nu}_j}^{\ast} \in H^{l_{1,j}}(W_f;R)$ in the property {\rm (\ref{prop:9.2})}, for any $k \neq l_{1,j}$ and for any element represented as
 $$(0,p) \in H^{n-k}(W_{f};R) \oplus R^{\sharp {\{1 \leq j \leq l_2 \mid l_{2,j}=k\}}} \cong H^{n-k}(W_{{f}^{\prime}};R)$$
 where $p$ is a sequence of
 $\sharp {\{1 \leq j \leq l_2 \mid l_{2,j}=k\}}$ integers such that exactly one number is $1$ and that the others are $0$, the product of the two elements ${i_{(f,f^{\prime}),S}}^{{\ast}^{\prime}}({{\nu}_j}^{\ast})$ and ${\phi}_{(f,f^{\prime}),S} \circ {\phi}_{(f,f^{\prime}),S,R,n-k}(p)$ in $H^{\ast}(W_{{f}^{\prime}};R)$ vanishes. 
\item
\label{prop:9.5}
Let $1 \leq j \leq l_2$ and $k:=l_{2,j}$. Take an element $$(0,p) \in H^{n-k}(W_{f};R) \oplus R^{\sharp {\{1 \leq j^{\prime \prime} \leq l_2 \mid l_{2,j^{\prime \prime}}=l_{2,j}=k\}}} \cong H^{n-k}(W_{{f}^{\prime}};R)$$ where $p$ is a sequence of
 $\sharp {\{1 \leq j^{\prime \prime} \leq l_2 \mid l_{2,j^{\prime \prime}}=l_{2,j}=k\}}$ integers such that exactly one number is $1$ and that the others are $0$. Let $N_{l,k}$ be the ascending sequence of all numbers of the set $\{1 \leq j^{\prime \prime} \leq l_2 \mid l_{2,j^{\prime \prime}}=l_{2,j}=k\}$.
 Assume also that $p$ is the element such that the $j^{\prime \prime \prime}$-th component is $1$ and that the $j^{\prime \prime \prime}$-th number of the sequence $N_{l,k}$ is $j$. For $j^{\prime} \in Q_j=\{1 \leq {j}^{\prime \prime} \leq l_1 \mid l_{2,j}=l_{1,{j}^{\prime \prime}}=k\}$, the product of the two elements ${i_{(f,f^{\prime}),S}}^{{\ast}^{\prime}}({{\nu}_{j^{\prime}}}^{\ast})$ and ${\phi}_{(f,f^{\prime}),S} \circ {\phi}_{(f,f^{\prime}),S,R,n-k}(p)$ in $H^{\ast}(W_{{f}^{\prime}};R)$ is represented as $n_{j,{j}^{\prime}}$ times the value ${\phi}_{(f,f^{\prime}),S} \circ {\phi}_{(f,f^{\prime}),S,R,n}(a)$ for a generator $a$ of the summand $R$ in the presented direct sum decomposition of $H^n(W_{{f}^{\prime}};R)$.
\end{enumerate}

\end{Prop}
\begin{proof}
The key ingredient in the proofs is that the Reeb space obtained by a bubbling operation to the original map can be regarded as a polyhedron obtained by attaching manifolds represented as connected sums of products of spheres along the (generating) polyhedron in a natural way (see FIGURE \ref{fig:9}) in knowing homology groups and cohomology rings: in knowing more precise topological information, we may not simply argue in this way. These explanations are also presented in the proof of Proposition \ref{prop:3}. 

The first three statements follow immediately from (the proofs of) Propositions \ref{prop:3} and \ref{prop:5} and related notions explained in Definitions \ref{def:4} , \ref{def:5} and \ref{def:6} and the explanation around them. 

\begin{figure}
\includegraphics[width=60mm]{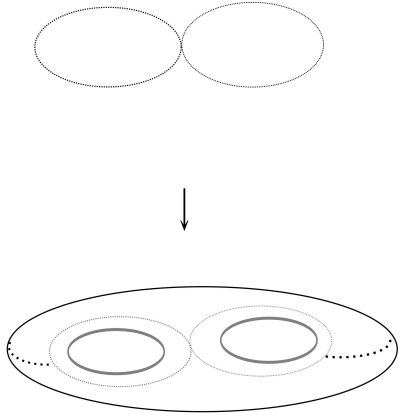}
\caption{A bouquet of two circles in the original Reeb space and a manifold represented as a connected sum of two copies of $S^1 \times S^{n-1}$ ($n=2$ in the figure) attached naturally along this: the bouquet is regarded as the generating polyhedron.}
\label{fig:9}
\end{figure}

For the proofs of the remaining three properties, we check several facts on structures of the modules. Each element of the summand $R^{\sharp {\{1 \leq j \leq l_2 \mid l_{2,j}=n-k\}}}$ of the module $H_k(W_{{f}^{\prime}};R) \cong H_k(W_{f};R) \oplus R^{\sharp {\{1 \leq j \leq l_2 \mid l_{2,j}=n-k\}}}$ is represented
 as a linear combination of classes represented by spheres intersecting once with $S_j$ ($1 \leq j \leq l_2$ satisfying the relation $l_{2,j}=n-k$), having no intersection with other $S_j$'s and regarded as fibers of the naturally existing normal and product bundles: see FIGUREs \ref{fig:10} and \ref{fig:11} and also the proof of Proposition \ref{prop:3}. Remember that ${\Sigma}_{j^{\prime} \in Q_j} n_{j,{j}^{\prime}}{\nu}_{{j}^{\prime}}$ is the class $S_j$ represents. The element $r_{o,j^{\prime \prime}}$ of the summand $R^{\sharp {\{1 \leq j \leq l_2 \mid l_{2,j}=n-k\}}}$ of the module $H^k(W_{{f}^{\prime}};R) \cong H^k(W_{f};R) \oplus R^{\sharp {\{1 \leq j \leq l_2 \mid l_{2,j}=n-k\}}}$ such that the $j^{\prime \prime}$-th entry is $1$ and that the other entries are $0$ is regarded as a cohomology class ${\phi}_{(f,f^{\prime}),S} \circ {\phi}_{(f,f^{\prime}),S,R,k}(r_{j^{\prime \prime}})$ in $H^{\ast}(W_{{f}^{\prime}};R)$ and it is characterized as a class satisfying the following properties.
\begin{enumerate}
\item At ${\phi}_{(f,f^{\prime}),S} \circ {\phi}_{(f,f^{\prime}),S,R,k}(r_{o,j^{\prime \prime}}) \in H_{k}(W_{{f}^{\prime}};R)$, the value is $1$.
\item Let $R_{j^{\prime \prime}}$ be the submodule of the summand $R^{\sharp {\{1 \leq j \leq l_2 \mid l_{2,j}=n-k\}}}$ generated by the set of all elements such that exactly one component is $1$ and that the other components are $0$ but $r_{o,j^{\prime \prime}}$.
On ${\phi}_{(f,f^{\prime}),S} \circ {\phi}_{(f,f^{\prime}),S,R,k}(R_{j^{\prime \prime}}) \subset H_{k}(W_{{f}^{\prime}};R)$, the value is $0$.
\item On ${i_{(f,f^{\prime}),S}}_{\ast}(H_{k}(W_{f};R))$, the value is $0$.
\end{enumerate}  

By this observation with (the proofs of) Propositions \ref{prop:3} and \ref{prop:5} and related notions explained in Definitions \ref{def:4} , \ref{def:5} and \ref{def:6} and the explanation around them as before, we can prove the last three statements.  
\begin{figure}
\includegraphics[width=60mm]{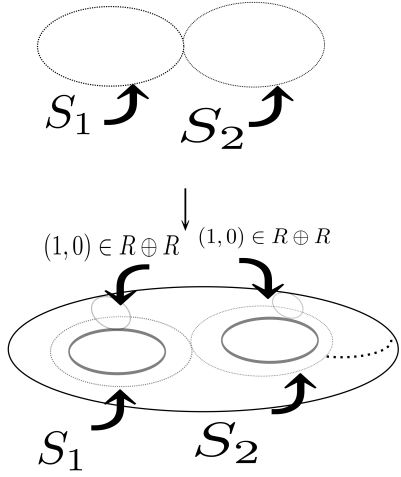}
\caption{The bouquet of two circles $S_1$ and $S_2$ and corresponding elements in $R \oplus R$ ($R^{\sharp {\{1 \leq j \leq l_2 \mid l_{2,j}=n-k\}}}$ in $H^k(W_{{f}^{\prime}};R) \cong H^k(W_{f};R) \oplus R^{\sharp {\{1 \leq j \leq l_2 \mid l_{2,j}=n-k\}}}$ for $(n,k)=(2,1)$).}
\label{fig:10}
\end{figure}
\begin{figure}
\includegraphics[width=60mm]{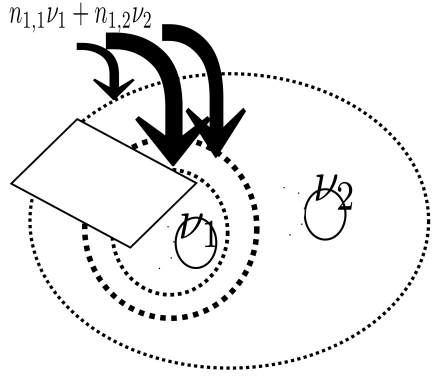}
\caption{A circle as a generating manifold and the class represented by this (a circle is represented by dotted curves and the square is for the abbreviation). ${\nu}_1$ and ${\nu}_2$ are classes represented by fundamental natural $1$-cycles in the Reeb space. The Reeb space is represented as a boundary connected sum of two copies of $S^1 \times D^{n-1}$. $n_{1,1}$ and $n_{1,2}$ are the coefficients.}
\label{fig:11}
\end{figure}

\end{proof}

For example, we can construct a map satisfying the assumption by a suitable default bubbling operation or a finite iteration of connected sums of maps obtained by suitable strongly trivial default normal bubbling operations.

By applying Proposition \ref{prop:9} one after another, we have the following.

\begin{Thm}
\label{thm:1}
Let $R$ be a PID having a unique identity element $1 \neq 0 \in R$. Let $m>n \geq 2$ be integers.
Let $f$ be a fold map on an $m$-dimensional closed and connected manifold $M$ into ${\mathbb{R}}^n$.
Let $\{G_j\}_{j=1}^{n}$ be a family of free finitely generated module over $R$ satisfying $G_1=\{0\}$ and $G_n \neq \{0\}$. 
Let $\{A_j\}_{j=1}^{{\rm rank} \quad G_n}$ be a family of sequences of non-negative integers whose lengths
 are all $n-1$ such that the sum of all the sequences and $\{{\rm rank} \quad G_j\}_{j=1}^{n-1}$ coincide where the sum of two sequences of numbers of a same length is defined in a natural way. We denote the $j_2$-th element of $A_{j_1}$ by $A_{j_1,j_2}$. 
Let $\{s_j\}_{j=1}^{n-1}$ be another family of non-negative integers. For any integer $1 \leq j_2 \leq n-1$, any integer $1 \leq j_3 \leq A_{j_1,j_2}$ and each integer $1 \leq j^{\prime} \leq s_{n-j_2}$,
 let $A_{j_1,j_2,j_3,{j}^{\prime}}$ be an arbitrary integer. 
In this situation, by the following steps, we can obtain a fold map ${f^{\prime}}$. \\
STEP 1 \\
Obtain a map satisfying the assumption of the original map in Proposition \ref{prop:9}. \\
STEP 2 \\
Consider a connected sum of $f$ and the map obtained in the previous step. \\
STEP 3 \\
To the map $f_1$ obtained in the previous step, perform a finite iteration of M-bubbling operations. \\
 \\
Moreover, we can obtain the map $f^{\prime}$ satisfying the following nine properties.
\begin{enumerate}
\item
\label{thm:1.1}
 There exist isomorphisms of modules $$H_k(W_{{f}^{\prime}};R) \cong H_k(W_{f};R) \oplus R^{s_k} \oplus G_k$$
 and $$H^k(W_{{f}^{\prime}};R) \cong H^k(W_{f};R) \oplus R^{s_k} \oplus G_k$$
for $1 \leq k \leq n-1$.
  There also exist isomorphisms of modules $H_n(W_{{f}^{\prime}};R) \cong H_n(W_{f};R) \oplus G_n$ and $H^n(W_{{f}^{\prime}};R) \cong H^n(W_{f};R) \oplus G_n$. 
Moreover, for $1 \leq k \leq n$, these isomorphisms are given by canonical homology isomorphisms and cohomology isomorphisms of the finite iteration of M-bubbling operations to obtain $f^{\prime}$ from $f_1$ {\rm :}they are defined in Definition \ref{def:6}.
\item
\label{thm:1.2}
 For $1 \leq k \leq n-1$, there exists an isomorphism ${i_{(f_1,f^{\prime}),R}}$ from $H_k(W_{f};R) \oplus R^{s_k}$ onto the domain of the inclusion morphism ${i_{(f_1,f^{\prime})}}_{\ast}:H_k(W_{f_1};R) \rightarrow H_k(W_{f^{\prime}};R)$ of the finite iteration of M-bubbling operations to obtain $f^{\prime}$ from $f_1$ and an isomorphism ${i_{(f_1,f^{\prime}),R}}$ from $H^k(W_{f};R) \oplus R^{s_k}$ onto the domain of the inclusion morphism ${i_{(f_1,f^{\prime})}}^{{\ast}^{\prime}}:H^k(W_{f_1};R) \rightarrow H^k(W_{f^{\prime}};R)$ of the finite iteration of M-bubbling operations to obtain $f^{\prime}$ from $f_1$ {\rm :}inclusion morphisms of the finite iterations of the M-operations are defined in Definition \ref{def:6} and $i_{(f_1,f^{\prime})}:W_{f_1} \rightarrow W_{f^{\prime}}$ is a canonical inclusion obtained by Lemma \ref{lem:1}.
\item
\label{thm:1.3}
  For $1 \leq k \leq n$, there exists an isomorphism ${{\phi}_{(f_1,f^{\prime}),R,k}}$ from $G_k$ onto the domain of the bubbling morphism ${{\phi}_{(f_1,f^{\prime})}}_{\ast}$ into $H_k(W_{f^{\prime}};R)$ of the finite iteration of M-bubbling operations to obtain $f^{\prime}$ from $f_1$ and an isomorphism ${{\phi}_{(f_1,f^{\prime}),R,k}}$ from $G_k$ onto the domain of the bubbling morphism ${{\phi}_{(f_1,f^{\prime})}}^{\ast}$ into $H^k(W_{f^{\prime}};R)$ of the finite iteration of M-bubbling operations to obtain $f^{\prime}$ from $f_1$ {\rm :}bubbling morphisms of the finite iterations of the M-operations are defined in Definition \ref{def:6}.
\item 
\label{thm:1.4}
The cohomology ring $H^{\ast}(W_{f_1};R)$ is isomorphic to a graded commutative algebra over $R$ obtained from the direct sum of the cohomology ring $H^{\ast}(W_{f};R)$ and a graded algebra such that the $k$-th module is isomorphic to $R^{s_k}$ for $1 \leq k \leq n-1$ and zero $k=0$ and $k \neq n$ and that the product of two elements of degree larger than $0$ always vanish and satisfying the following four {\rm :}we identify $H^{\ast}(W_{f_1};R)$ with this algebra if it is not confusing. 
\begin{enumerate}
\item For an integer $i>0$, the $i$-th module is the direct sum of the $i$-th modules of the summands of the direct sum.
\item The $0$-th module is $R$.
\item Take a pair $(a_{i_1,1},a_{i_1,2})$ of elements of degree $i_1>0$ in the direct sum and a pair $(a_{i_2,1},a_{i_2,2})$ of elements of degree $i_2>0$ in the direct sum, which are elements of $i_1$ and $i_2$ in the direct sum, respectively. The product is $(a_{i_1,1}a_{i_2,1},a_{i_1,2}a_{i_2,2})$ and of degree $i_1+i_2$.
\item For $r \in R$, which is also an element of degree $0$, and a pair $(a_{i,1},a_{i,2})$ of elements of degree $i>0$, which is an element of degree $i$ in the direct sum, the product is $(ra_{i,1},ra_{i,2})$ and of degree $i$.
\end{enumerate} 

 A similar fact holds for the module $H_{k}(W_{f_1};R)$: we do not define products for this.
\item
\label{thm:1.5}
 For an arbitrary positive integer $k_1 \leq n-1$, we restrict the module $H^{k_1}(W_{{f}^{\prime}};R) \cong H^{k_1}(W_{f};R) \oplus R^{s_{k_1}} \oplus G_{k_1}$ to the image of ${i_{(f_1,f^{\prime})}}^{{\ast}^{\prime}}:H^{k_1}(W_{f_1};R) \rightarrow H^{k_1}(W_{f^{\prime}};R)$ and in addition according to {\rm (\ref{thm:1.4})} just before restrict to the image of $H^{k_1}(W_{f};R) \oplus \{0\} \subset H^{k_1}(W_{f};R) \oplus R_{s_{k_1}}$, identified with $H^{k_1}(W_{f_1};R)$. After that we take an element $a_1$ in this set. For an arbitrary positive integer $k_2 \leq n-1$, we restrict the module $H^{k_2}(W_{{f}^{\prime}};R) \cong H^{k_2}(W_{f};R) \oplus R^{s_{k_2}} \oplus G_{k_2}$ to the submodule generated by ${i_{(f_1,f^{\prime})}}^{{\ast}^{\prime}} \circ {i_{(f_1,f^{\prime}),R}}(\{0\} \oplus R^{s_{k_2}}) \bigcup {{\phi}_{(f_1,f^{\prime}),R}}_{\ast} \circ {\phi}_{(f_1,f^{\prime}),R,k_2}(G_{k_2})$ and take an element $a_2$ of this. In this situation, the product of $a_1$ and $a_2$ in $H^{\ast}(W_{{f}^{\prime}};R)$ vanishes.  

\item
\label{thm:1.6}
For any element of $G_{k_j}$ in $$H^{k_1}(W_{{f}^{\prime}};R) \cong H^{k_1}(W_{f};R) \oplus R^{s_{k_j}} \oplus G_{k_j},$$ take an element
 $a_j \in {{\phi}_{(f_1,f^{\prime}),R}}_{\ast} \circ {{\phi}_{(f_1,f^{\prime}),R,k_1}}(G_{k_j})$ for $j=1,2$.
 In this situation, the product of $a_1$ and $
a_2$ vanishes.
\item 
\label{thm:1.7}
We may identify $G_{k}$ with an $R$-module ${\oplus}_{j=1}^{{\rm rank} \quad G_n} R^{A_{j,k}}$ by a suitable isomorphism {\rm :} we abuse this identification in the following two properties if it is not confusing.
\item
\label{thm:1.8}
Take an element $a$ of $G_{k_1}$ {\rm (}$1 \leq k_1 \leq n-1${\rm )}. The product of a non-zero element in $H^{k}(W_{{f}^{\prime}};R)$ satisfying $k \neq 0,n-k_1$ and ${{\phi}_{(f_1,f^{\prime}),R}}_{\ast} \circ {{\phi}_{(f_1,f^{\prime}),R,k_1}}(a)$ vanishes. 

\item
\label{thm:1.9} 
Take an element $a$ of ${\oplus}_{j=1}^{{\rm rank} \quad G_n} R^{A_{j,k_1}} \cong G_{k_1}$ {\rm (}$1 \leq k_1 \leq n-1${\rm )} such that the $k_2$-th entry in the summand $R^{A_{j,k_1}}$ is $1$ and the other entries are $0$ for any other summand in the direct sum decomposition.
 For any element in $H^{n-k_1}(W_{{f}^{\prime}};R)$ of the form ${i_{(f_1,f^{\prime})}}^{{\ast}^{\prime}} \circ {i_{(f_1,f^{\prime}),R}}((0,p,0))$ for $$(0,p,0) \in H^{n-k_1}(W_{f};R) \oplus R^{s_{n-k_1}} \oplus G_{n-k_1}$$ where the $k_3$-th entry of $p$ is $1$ and the other entries are $0$, consider the product with the element ${{\phi}_{(f_1,f^{\prime}),R}}_{\ast} \circ {{\phi}_{(f_1,f^{\prime}),R,k_1}}(a)$ for $a$ before. By this, we have an element is $A_{j,k_1,k_2,k_3}$ times the value ${{\phi}_{(f_1,f^{\prime}),R}}_{\ast} \circ {{\phi}_{(f_1,f^{\prime}),R,n}}(a_j)$ at a generator $a_j$ of the $j$-th summand of $G_n$ in the direct sum decomposition of $G_n$, which is considered in the presented direct sum decomposition of $H^{n}(W_{{f}^{\prime}};R)$.
\end{enumerate}
Furthermore, a map constructed in STEP 1 can be special generic and can be replaced by an arbitrary map obtained by a finite iteration of bubbling operations which are not M-bubbling operations and whose generating polyhedra are of dimensions smaller than $n-1$ to such a special generic map,
\end{Thm}

Rigorous proofs are left to readers and we present a sketch of the proof.

\begin{proof}[A sketch of the proof]
For STEP 1 and STEP 2, we consider a connected sum of the map $f$ and a map obtained by a default trivial bubbling operation whose generating polyhedron is a
 bouquet consisting of $s_j$ $j$-dimensional standard spheres ($1 \leq j \leq n-2$), or equivalently as explained in Example \ref{ex:3} by a finite iteration of default strongly trivial normal bubbling operations and connected sums of them: we can consider a more general map for the latter map satisfying the assumption of Proposition \ref{prop:9} by virtue of Lemma \ref{lem:1} and the topologies of Reeb spaces obtained by M-bubbling operations. This proves the last part of the statement. 

The proof of the main nine properties is completed by applying Proposition \ref{prop:9} one after another. 
First, the first five properties are shown by using the properties (\ref{prop:9.1}), (\ref{prop:9.2}) and (\ref{prop:9.3.0}) of Proposition \ref{prop:9} and observing the topological structures of the Reeb spaces. 

The sixth property is owing to the property (\ref{prop:9.3}) of Proposition \ref{prop:9}.

We explain about the remaining properties. We implicitly use the properties (\ref{prop:9.4}) and (\ref{prop:9.5}) of Proposition \ref{prop:9}.

In STEP 3, the time of M-bubbling operations we perform is ${\rm rank} \quad G_n$. 
At each step or the stage of the $j$-th operation, we perform an M-bubbling operation whose
 generating polyhedron is a bouquet consisting of $A_{j,k_1}$ ($n-k_1$)-dimensional standard spheres for $2 \leq k_1 \leq n-1$. Moreover, for the $k_2$-th ($n-k_1$)-dimensional standard sphere in the $A_{j,k_1}$ spheres here, the homology class it represents is a linear combination of the classes represented by $s_{n-k_1}$ ($n-k_1$)-dimensional standard spheres as in Propositions \ref{prop:8} and \ref{prop:9} and the coefficient at the $k_3$-th class in the family of the classes is $A_{j,k_1,k_2,k_3}$.  
\end{proof}

\begin{Rem}
\label{rem:1}
In the situation of Proposition \ref{prop:9} and Theorem \ref{thm:1}, let $R$ be $\mathbb{Z}$ or $\mathbb{Q}$. 

According to this work, cohomology rings of Reeb spaces seem to be various. In fact, as an easy observation, by a bubbling operation of Theorem \ref{thm:1} to the map explained just before, we can obtain cohomology rings of Reeb spaces such that there exists no pair of cocycles satisfying the following properties.

\begin{enumerate}
\item The cocycles are not $0$-cycles.
\item The class represented by the product of these two cocycles does not vanish. 
\end{enumerate}     

We can also obtain ones such that there exist such pairs by virtue of Proposition \ref{prop:9} (\ref{prop:9.5}) and Theorem \ref{thm:1} (\ref{thm:1.9}). 

Let us change a non-zero coefficient number of the cycle or the class represented by a sphere whose dimension is not so large in the generating polyhedron into another non-zero number. By this step, we cannot change the resulting cohomology ring whose coefficient ring is $\mathbb{Q}$. On the other hand, the resulting cohomology ring whose coefficient ring is $\mathbb{Z}$ changes (in general).

In short, we can easily obtain families of Reeb spaces such that cohomology rings are mutually isomorphic in the cases where $R=\mathbb{Q}$ and that cohomology rings are mutually non-isomorphic in the cases where $R=\mathbb{Z}$.

Proposition \ref{prop:7} implies that from cohomology rings of Reeb spaces, we can know the cohomology rings of the $m$-dimensional manifolds
 considerably in several cases.

These facts explicitly show that difference of the topologies (cohomology rings) of Reeb spaces of fold maps of a suitable class are closely related to difference of the topologies (cohomology rings) of manifolds admitting these maps in several explicit situations. 

According to \cite{saeki}, \cite{saeki2}, \cite{saekisakuma} and \cite{wrazidlo}, it is explicitly found that in considerable cases special generic maps restrict the topologies and the differentiable structures of the manifolds strongly. For example, in considerable cases, exotic homotopy spheres do not admit special generic maps into Euclidean spaces whose dimensions are larger than $2$, where homotopy spheres except exotic $4$-dimensional spheres, being undiscovered, admit special generic maps into $\mathbb{R}$ and (in cases where the dimensions of the homotopy spheres are larger than $1$) ${\mathbb{R}}^2$. Later, on $7$-dimensional homotopy spheres, stable fold maps into ${\mathbb{R}}^4$ such that the singular value sets are concentric spheres (or {\it round} fold maps, introduced in \cite{kitazawa3}), that preimages of regular values are disjoint unions of standard spheres and that satisfy the assumption of Proposition \ref{prop:7}, were constructed by the author in \cite{kitazawa} and \cite{kitazawa2}. The author also explicitly found that the numbers of connected components of the singular value sets are closely related to the differentiable structures of the homotopy spheres. As a new work, in the present paper, we first demonstrated a similar work related to cohomology rings of manifolds admitting fold maps of explicit classes. 
\end{Rem}

\begin{Thm}
\label{thm:2}
In the situation of Theorem \ref{thm:1}, if $m$ is sufficiently large, then in STEP 1, we do not need to use special generic maps obtained by default trivial bubbling operations whose generating polyhedra are bouquets consisting of standard spheres or we do not need to assume "trivial".
\end{Thm}
\begin{proof}
By an S-bubbling operation (a finite iteration of S-bubbling operations), as presented in FIGURE \ref{fig:5} for example, locally we obtain a connected component containing exactly one singular point of the preimage of a segment $I$ by $q_{f^{\prime}}$ satisfying the following three where we abuse the notation in Theorem \ref{thm:1}.
\begin{enumerate}
\item The value of the quotient map $q_{f^{\prime}}$ to the Reeb space at the singular point is contained in the interior of the segment $I$. 
\item If we remove the point or the value just before from the segment $I$, then the resulting space consists of exactly two connected components and they are in different connected components of $W_{f^{\prime}}-q_{f^{\prime}}(S(f^{\prime}))$ for the resulting map $f^{\prime}$: we can use terminologies on {\it transversalities} of curves or smooth manifolds and smooth maps, however, we do not use.
\item For the map $q_{f^{\prime}}$, the value is a non-manifold point of the Reeb space and for a small neighborhood of it and points there which are not values of singular points, the preimages are standard spheres. Moreover, the preimage of the segment $I$ is diffeomorphic to a manifold obtained by removing the interior of a disjoint union of three copies of a standard closed ($m-n+1$)-dimensional disc smoothly embedded in $S^{m-n+1}$. 
\end{enumerate}
See also FIGURE \ref{fig:12}, representing the map ${q_{f^{\prime}}} {\mid}_{{q_{f^{\prime}}}^{-1}(I)}:{q_{f^{\prime}}}^{-1}(I) \rightarrow I$ and the map obtained by composing $\bar{f^{\prime}}$, which we can regard as a function: arrows represent a diffeomorphism on the preimage preserving the value of the local map to the Reeb space at each point and the value of the resulting local function at each point. The assumption that $m$ is sufficiently large produces this symmetry (in Remark \ref{rem:2} this is explained in a more precise way). We can see that we do not need to assume "trivial" and this and the argument on construction of fold maps via S-bubbling operations, which we will explain in the last part of Remark \ref{rem:2}, complete the proof.

\begin{figure}
\includegraphics[width=40mm]{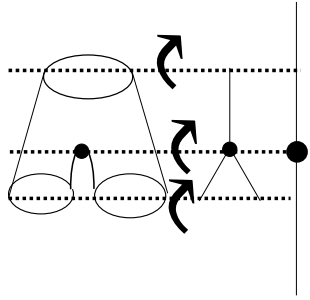}
\caption{The preimage of a small natural Y-shaped $1$-dimensional polyhedron containing a value of the quotient map to 
the Reeb space, which is the Y-shaped polyhedron, at the exactly one singular point and the natural function; the value is the exactly one non-manifold point in the interior of the Y-shaped polyhedron. Preimages of remaining points are standard spheres. Arrows are for a diffeomorphism preserving the value of the local map to the polyhedron at each point and the value of the natural local function at each point.}

\label{fig:12}
\end{figure}
\end{proof}
\begin{Rem}
\label{rem:2}
In the situation of Theorem \ref{thm:2}, we can construct families of manifolds such that characteristic classes represented by cohomology classes of the manifolds such as Stiefel-Whitney classes, Pontryagin classes, and so on, are mutually distinct and that the Reeb spaces of the fold maps are homeomorphic. As a result, cohomology rings of the Reeb spaces are isomorphic, and if we can apply Proposition \ref{prop:7}, then those of the manifolds seem to be similar even if they are not isomorphic. To obtain maps on manifolds such that these classes (do not) vanish, we construct a special generic map first so that the normal bundles of connected components singular sets have vanishing (non-vanishing) Stiefel-Whitney classes, Pontryagin classes, and so on. After that, as M-bubbling operations, we perform S-bubbling operations preimages of points in whose generating polyhedra are standard spheres to the map without bearing new non-vanishing Stiefel-Whitney classes(, Pontryagin classes, and so on), for example. 

Throughout the procedure, the assumption that $m$ and $m-n$ are large is essential: for vector bundles or linear bundles dimensions of whose fibers are large, structure groups, acting on the fibers, can be reduced and the dimensions of the fixed point s
ets of the fibers are sufficiently large. It is based on the well-known obstruction theory. For fundamental and classical theory of characteristic classes of vector (linear) bundles, tangent bundles and differentiable manifolds, see \cite{milnorstasheff} for example.
\end{Rem}
\section{A general version of Theorem \ref{thm:1}.}
\label{sec:4}
\begin{Def}
\label{def:7}
A manifold $S$ is said to be {\it CPS} if either of the following hold.
\begin{enumerate}
\item $S$ is a standard sphere whose dimension is positive.
\item $S$ is represented as a connected sum or a product of two CPS manifolds.
\end{enumerate}
\end{Def}
We can know the following by virtue of fundamental differential topological discussions and omit the proof.
\begin{Prop}
\label{prop:10}
CPS manifolds can be embedded into one-dimensional higher Euclidean spaces. CPS manifolds also admit diffeomorphisms reversing the orientations of the manifolds.
\end{Prop}

\begin{Def}
\label{def:8}
A graded commutative algebra $A$ over a PID $R$ is said to be {\it CPS} if either of the following three hold.
\begin{enumerate}
\item $A$ is isomorphic to the cohomology ring $H^{\ast}(S^k;R)$ ($k \geq 1$).
\item $A$ is represented as a tensor product of two CPS graded commutative algebras over $R$.
\item $A$ is a graded commutative algebra obtained from two CPS graded commutative algebras $A_1$ and $A_2$ over $R$ satisfying the following properties.
\begin{enumerate}
\item The maximal degrees of $A_1$ and $A_2$ are $d>0$.
\item The $d$-th modules of $A_1$ and $A_2$ are of rank $1$ and free.
\end{enumerate}
Moreover, $A$ is obtained by the following steps.
\begin{enumerate}
\item Choose an one element set $\{a_i\} \subset A_i$ generating the $d$-th module of $A_i$.
\item We set a graded module $A$ over $R$ as the following rules.
\begin{enumerate}
\item For an integer $i>0$ satisfying $i \neq d$, the $i$-th module is the direct sum of the $i$-th module of $A_1$ and the $i$-th module of $A_2$.
\item The $d$-th module is a free module of rank $1$ over $R$ obtained from the direct sum of the $d$-th module of $A_1$ and the $d$-th module of $A_2$ by regarding $(a_1,a_2)$ in the direct sum as $0$.
\item The $0$-th module is $R$.
\end{enumerate}
We set $A$ as a graded commutative algebra over $R$ satisfying the following rules.
\begin{enumerate}
\item For a pair $(a_{i_1,1},a_{i_1,2}) \in A_1 \oplus A_2$ of elements of degree $i_1>0$ and a pair $(a_{i_2,1},a_{i_2,2}) \in A_1 \oplus A_2$ of elements of degree $i_2>0$, which are elements of degree $i_1$ and degree $i_2$, respectively, the product is $(a_{i_1,1}a_{i_2,1},a_{i_1,2}a_{i_2,2}) \in A_1 \oplus A_2$ and of degree $i_1+i_2$ if $i_1+i_2 \neq d$ and if $i_1+i_2=d$, then we map the element by the quotient map onto the free module of rank $1$ explained in the previous explanation.
\item For $r \in R$, which is also an element of degree $0$ and a pair $(a_{i,1},a_{i,2}) \in A_1 \oplus A_2$ of elements
 of degree $i>0$, the product is $(ra_{i,1},ra_{i,2}) \in A_1 \oplus A_2$ and of degree $i$ if $i \neq d$ and if $i=d$, then we map the element by the quotient map onto the free module of rank $1$ explained in the previous explanation.
\end{enumerate} 
\end{enumerate}
\end{enumerate}
\end{Def}
The following proposition gives a natural one-to-one correspondence and we can know this easily.
\begin{Prop}
\label{prop:11}
We can canonically correspond a CPS graded commutative algebra to a CPS manifold by taking its cohomology ring and we can consider a natural converse correspondence. 

\end{Prop}

Motivated by Proposition \ref{prop:11}, we can define a new class of graded commutative algebras.

\begin{Def}
A graded commutative algebra $A$ over a PID $R$ is said to be {\it GCPS} if it is isomorphic to a cohomology ring whose coefficient ring is $R$ of a bouquet of a finite number of CPS manifolds: the number is in general positive and if it is zero then the bouquet is replaced by a single point.
\end{Def}
The following is an extension of Proposition \ref{prop:8} and we omit the proof. This is also done by applying fundamental differential topological discussions respecting the definition and the structure of a CPS manifold and so on. The short explanation after Proposition \ref{prop:8} is also a key. 
More rigorous proofs are left to readers.
\begin{Prop}
\label{prop:12}
We assume the following four conditions.

\begin{enumerate}
\item $n \geq 2$ is a positive integer.
\item $Q$ is an $n$-dimensional, comapct and connected manifold we can embed smoothly into ${\mathbb{R}}^n$ and for each integer $1 \leq j \leq n-1$, we denote by $r_j$ the rank of the subgroup $G_{Q,j}$ of the $j$-th homology group $H_j(Q;\mathbb{Z})$ generated by the set of a finite number of elements $a_{Q,j,q_j}$ satisfying the following properties where $1 \leq q_j \leq r_j$ is an integer.
\begin{enumerate}
\item We can take duals of $a_{Q,j,q_j}$ where the coefficient ring is $R$. 
\item $a_{Q,j,q_j}$ is represented by a $j$-dimensional standard sphere smoothly embedded in the interior of ${\rm Int} Q$. 
\end{enumerate}
Set $l_1:={\Sigma}_{j=1}^{n-1} r_j$. 
To all integers $1 \leq j^{\prime} \leq l_1$, we assign an integer $1 \leq l_{1,j^{\prime}} \leq n-1$ satisfying $$\sharp \{ 1 \leq j^{\prime} \leq l_1 \mid l_{1,j^{\prime}}=j\}=r_j.$$
\item $P$ is a compact polyhedron represented as a bouquet of $l_2 \geq 0$ standard spheres
 the dimension of each of which is greater than or equal to $1$ and smaller than or equal to $n-2$: if $l_2=0$, then let $P$ be a point. Let $j$ be an integer satisfying the relation $1 \leq j \leq l_2$ and $l_{2,j}$ be an integer satisfying the relation $1 \leq l_{2,j} \leq n-2$. Each of these $l_2$ spheres is denoted by $S_j$ and diffeomorphic to $S^{l_{2,j}}$. 
\item For each integer $1 \leq j \leq l_2$, let $Q_j$ be the set of all integers
 $1 \leq {j}^{\prime} \leq l_1$ satisfying $l_{2,j}=l_{1,{j}^{\prime}}$. We take an arbitrary integer $n_{j,{j}^{\prime}}$ for $j^{\prime} \in Q_j$.
\end{enumerate}
In this situation, we can realize $P$ as a polyhedron in $Q$ {\rm (}we denote this by $S$ and each corresponding sphere by $S_j$ samely{\rm )} so that the following properties hold.
\begin{enumerate}
\item $S$ is in ${\rm Int} Q$ and each of the $l_2$ spheres $S_j$ is regarded as a closed submanifold of $Q$.
\item For each integer $1 \leq {j}^{\prime} \leq l_1$, we can set
 a class ${\nu}_{{j}^{\prime}} \in G_{Q,l_{1,j^{\prime}}}$ represented by a sphere $S^{l_{1,j^{\prime}}} \times \{0\} \subset S^{l_{1,j^{\prime}}} \times {\rm Int} D^{n-l_{1,{j}^{\prime}}} \subset S^{l_{1,j^{\prime}}} \times D^{n-l_{1,{j}^{\prime}}} \subset Q$ so that the following two properties hold.
\begin{enumerate}
\item These classes form bases for each group $G_{Q,j}$ {\rm(}$1 \leq j \leq n-1${\rm )}.
\item $S_j$ represents the class ${\Sigma}_{j^{\prime} \in Q_j} n_{j,{j}^{\prime}}{\nu}_{{j}^{\prime}}$. 
\end{enumerate}
\end{enumerate}
\end{Prop}

Note that an $n$-dimensional compact and connected manifold represented as a regular neighborhood of a boundary connected sum of a finite number of manifolds each of which is represented as a product of a CPS manifold and a standard closed disc is an explicit manifold of the class of the $n$-dimensional manifolds we can embed into ${\mathbb{R}}^n$. 

We have the following theorem as an extension of (Proposition \ref{prop:9} and) Theorem \ref{thm:1} for an explicit situation.

\begin{Thm}
\label{thm:3}
Let $R$ be a PID having the identity element $1 \neq 0 \in R$.

Let $m>n \geq 2$ be integers.

Let $f$ be a fold map on an $m$-dimensional closed and connected $M$ into ${\mathbb{R}}^n$.
Let $\{G_j\}_{j=1}^{n}$ be a family of free finitely generated modules over $R$ satisfying two relations $G_1=\{0\}$ and $G_n = R$. 

Let $A$ be a graded commutative algebra over $R$ isomorphic to the cohomology ring of a compact and connected manifold $Q$ of dimension $n$ we can smoothly embed into ${\mathbb{R}}^n$ whose coefficient ring is $R$. We assume that by fixing a suitable identification between the graded commutative algebras, $B_k$ is the internal direct sum of the submodule $B_{R,s_k}$, which will be explained in the following, and a suitable submodule $T_k$ over $R$ where $B_k$ denotes the $k$-th module of $A$. $B_{R,s_k}$ is defined by the following.

Let $\{s_j\}_{j=1}^{n-1}$ be the sequence where $s_j$ denotes the rank of the submodule of $B_k$ defined as the module identified with the submodule of $H^k(Q;R)$ satisfying the following two. We denote this module by $B_{R,s_k}$. 
\begin{enumerate}
\item It is generated by the set of all classes regarded as duals of $k$-th homology classes represented by smoothly embedded $k$-dimensional standard spheres in ${\rm Int} Q$.
\item It is free.
\end{enumerate}

We abuse these identifications and notation in this theorem.

For any integer $1 \leq j \leq n-1$,  
 any integer $1 \leq j_1 \leq {\rm rank} \quad G_j$ and each integer $1 \leq j_2 \leq s_{n-j}$,
 let $A_{1,j,j_1,j_2}$ be an arbitrary integer. 

\noindent In this situation, by the following steps, we can obtain a fold map ${f^{\prime}}$.\\
STEP 1\\
Obtain a map by a default bubbling operation. \\
STEP 2\\
Consider a connected sum of $f$ and the map obtained in the previous step. \\
STEP 3\\
To the map $f_1$ obtained in the previous step, perform an M-bubbling operation. \\
\\
Moreover, we can obtain the map $f^{\prime}$ satisfying the following properties: identifications of modules and so on are similar to ones in Theorem \ref{thm:1}. 
\begin{enumerate}
\item Properties presented in {\rm (\ref{thm:1.1})} of Theorem \ref{thm:1} where for two isomorphisms of modules "$R^{s_k}$" is replaced by $B_k$.
\item Properties presented in {\rm (\ref{thm:1.2})} of Theorem \ref{thm:1} where "$R^{s_k}$" is replaced by $B_k$.
\item Properties presented in {\rm (\ref{thm:1.3})} of Theorem \ref{thm:1}.
\item The cohomology ring $H^{\ast}(W_{f_1};R)$ is obtained from a direct sum of the cohomology ring $H^{\ast}(W_{f};R)$ and
 a graded commutative algebra isomorphic to $A$ in a similar way presented in {\rm (\ref{thm:1.4})} of Theorem \ref{thm:1}.
\item Properties presented in {\rm (\ref{thm:1.5})} and {\rm (}\ref{thm:1.6}{\rm )} of Theorem \ref{thm:1} where "$R^{s_k}$" is replaced by $B_k$.
\item Properties presented in {\rm (\ref{thm:1.7})} of Theorem \ref{thm:1} where the direct sum of "$R^{A_{j,k}}$"'s is replaced by a suitable free module over $R$ with ${\rm rank} \quad G_n=1$: we denote this by $R_{0,k}$.
\item Properties presented in {\rm (\ref{thm:1.8})} of Theorem \ref{thm:1}.
\item Properties similar to ones presented in {\rm (\ref{thm:1.9})} of Theorem \ref{thm:1}. Take an element $a$ of $R_{0,j}$ {\rm (}$1 \leq j \leq n-1${\rm )} before such that the $j_1$-th entry of $R_{0,j}$ is $1$ and that the other entries are $0$.
 For any element in $H^{n-j}(W_{{f}^{\prime}};R)$ of the form ${i_{(f_1,f^{\prime})}}^{{\ast}^{\prime}} \circ {i_{(f_1,f^{\prime}),R}}((0,p,0,0))$ for $$(0,p,0,0) \in H^{n-j}(W_{f};R) \oplus B_{R,s_{n-j}} \oplus \{0\} \oplus G_{j}$$ \\
      $$\subset H^{n-j}(W_{f};R) \oplus B_{n-j} \oplus G_{n-j}$$ where the $j_2$-th entry of $p$ is $1$ and the others are $0$, consider the product with the element ${{\phi}_{(f_1,f^{\prime}),R}}_{\ast} \circ {{\phi}_{(f_1,f^{\prime}),R,k_1}}(a)$ for $a$ before. By this, we have an element, which is $A_{1,j,j_1,j_2}$ times the value ${{\phi}_{(f_1,f^{\prime}),R}}_{\ast} \circ {{\phi}_{(f_1,f^{\prime}),R,n}}(a_j)$ at a generator of $G_n$ in the direct sum decomposition of $H^{n}(W_{{f}^{\prime}};R)$.
\item 
 For any element in $H^{n-k_1}(W_{{f}^{\prime}};R)$ of the form ${i_{(f_1,f^{\prime})}}^{{\ast}^{\prime}} \circ {i_{(f_1,f^{\prime}),R}}((0,0,p,0))$ for $$(0,0,p,0) \in H^{n-k_1}(W_{f};R) \oplus \{0\} \oplus T_{n-k_1} \oplus G_{n-k_1}$$ \\
$$\subset H^{n-k_1}(W_{f};R) \oplus B_{n-k_1} \oplus G_{n-k_1}$$ and any element of $R_{0,k_1} \cong G_{k_1}$ before, the product is zero.
\end{enumerate}

\end{Thm}
We present important ingredients of the proof. Rigorous proofs are left to readers.
\begin{proof}[Important ingredients of the proof.]
By the assumption, we can obtain a special generic map such that the cohomology ring of the Reeb space is isomorphic to $A$ and that the map obtained by the restriction to the singular set is an embedding. 

This completes the STEP 1 and is a key to the first property.

We can prove the nine listed properties and all we need to show in a way similar to that of the proof of Theorem \ref{thm:1}.
For example, Proposition \ref{prop:12} is essential and plays a role Proposition \ref{prop:9} does in the proof of Theorem \ref{thm:1}.
\end{proof}

As $A$, the class of GCPS graded commutative algebras over $R$ is an explicit important class for example. 

We can also obtain a natural generalization of Theorem \ref{thm:2} and comment as Remarks \ref{rem:1} and \ref{rem:2} on Theorem \ref{thm:3}.

\section{Acknowledgement}
\thanks{The author is a member of the project Grant-in-Aid for Scientific Research (S) (17H06128 Principal Investigator: Osamu Saeki)

"Innovative research of geometric topology and singularities of differentiable mappings"
(https://kaken.nii.ac.jp/en/grant/KAKENHI-PROJECT-17H06128/
).} and supported by this.

\end{document}